\documentclass[11pt]{article}
\usepackage{graphicx,array,amssymb,amsmath,psfrag,fullpage,multirow}
\usepackage{psfrag,graphics,epsfig,multirow,multicol,times,color,soul}
\usepackage{algorithm}
\usepackage{algorithmic}
\usepackage{enumitem}

\newcommand{\BEAS}{\begin{eqnarray*}}
\newcommand{\EEAS}{\end{eqnarray*}}
\newcommand{\BEA}{\begin{eqnarray}}
\newcommand{\EEA}{\end{eqnarray}}
\newcommand{\BEQ}{\begin{equation}}
\newcommand{\EEQ}{\end{equation}}
\newcommand{\BIT}{\begin{itemize}}
\newcommand{\EIT}{\end{itemize}}
\newcommand{\BNUM}{\begin{enumerate}}
\newcommand{\ENUM}{\end{enumerate}}

\newcommand{\BA}{\begin{array}}
\newcommand{\EA}{\end{array}}
\newcommand{\BC}{\begin{center}}
\newcommand{\EC}{\end{center}}

\newcommand{\reals}{{\mbox{\bf R}}}

\newcommand{\symm}{{\mbox{\bf S}}}  % symmetric matrices

\newcommand{\sgn}{\mbox{sgn}}

\newcommand{\QED}{~~\rule[-1pt]{6pt}{6pt}}

\newcommand{\argmin}{\mathop{\rm argmin}}

\newcommand{\argmax}{\mathop{\rm argmax}}

\newtheorem{theorem}{Theorem}

\newtheorem{example}[theorem]{Example}
\newtheorem{proposition}[theorem]{Proposition}

\newtheorem{remark}[theorem]{Remark}
\newtheorem{lemma}[theorem]{Lemma}

\newtheorem{definition}[theorem]{Definition}
\newcounter{exno}

\newenvironment{proof}{\textbf{Proof.}}{\QED\bigskip}

{\begin{quote}}{\end{quote}}

\makeatletter
\long\def\@makecaption#1#2{
   \vskip 9pt
   \begin{small}
   \setbox\@tempboxa\hbox{{\bf #1:} #2}
   \ifdim \wd\@tempboxa > 5.5in
        \begin{center}
        \begin{minipage}[t]{5.5in}
        \addtolength{\baselineskip}{-0.95pt}
        {\bf #1:} #2 \par
        \addtolength{\baselineskip}{0.95pt}
        \end{minipage}
        \end{center}
   \else
    \hbox to\hsize{\hfil\box\@tempboxa\hfil}
   \fi
   \end{small}\par
}
\makeatother

\newcounter{oursection}

\newcounter{lecture}

\title{Conditional Gradient Algorithms for Rank-One Matrix Approximations with a Sparsity Constraint}

\author{Ronny Luss\thanks{School of Mathematical Sciences, Tel Aviv University, Ramat Aviv 69978, Israel ({\tt ronnyluss@gmail.com}).}
\and Marc Teboulle\thanks{School of Mathematical Sciences, Tel Aviv University, Ramat Aviv 69978, Israel ({\tt teboulle@post.tau.ac.il}).}}

\begin{document}

 \begin{center}

 {\Large\bf Conditional Gradient Algorithms for Rank-One Matrix Approximations with a Sparsity Constraint}

\bigskip

  Ronny Luss\footnote{Corresponding author} and Marc Teboulle\medskip

School of Mathematical Sciences\\
Tel-Aviv University, Ramat-Aviv 69978, Israel\\ email: {\texttt ronnyluss@gmail.com, teboulle@post.tau.ac.il}\\
\medskip

\today

\end{center}
\medskip

\begin{abstract}
The sparsity constrained rank-one matrix approximation problem is a difficult mathematical optimization problem which arises in a wide array of useful applications in engineering, machine learning and statistics, and the design of algorithms for this problem has attracted
intensive research activities. We introduce an algorithmic framework, called ConGradU, that
unifies a variety of seemingly different algorithms that have been derived from disparate
approaches, and allows for deriving new schemes. Building on the old and well-known
conditional gradient algorithm, ConGradU is a simplified version with unit step size and
yields a generic algorithm which either is given by an analytic formula or requires a very low
computational complexity. Mathematical properties are systematically developed and numerical
experiments are given.
\end{abstract}


\section{Introduction}
The problem of interest here is the sparsity constrained rank-one matrix approximation given
by
\begin{equation} \label{eq:sparse_pca_intro}
\max \{ x^TAx : \|x\|_2=1, \|x\|_0\leq k , x\in \reals^n \},
\end{equation}
where $A\in\symm^n$ is a given real symmetric matrix, and $1 <k \leq n$ is a parameter controlling the sparsity of $x$ which is
defined by counting the number of nonzero entries of $x$ and denoted using the $l_0$ notation: $\|x\|_0=|\{i: x_i \neq  0\}|$. This problem is also commonly known as the sparse Principal Component Analysis (PCA) problem, or as we refer to it, $l_0$-constrained PCA. Without the $l_0$ constraint,
the problem reduces to finding the first principal eigenvector and the corresponding maximal eigenvalue of
the matrix $A$, i.e., solves
$$
\max \{x^TAx : \|x\|_2=1, \; x\in\reals^n \},
$$
which is the PCA problem.

Suppose $A=B^TB$ where $B$ is an $m\times n$ mean-centered data matrix with $m$ samples and $n$ variables, and denote by $v$ be the principal eigenvector of $A$, i.e., $v$ solves the above PCA problem.  Then $Bv$ projects the data $B$ to one dimension that maximizes the variance of the projected data.  In general, PCA can be used to reduce $B$ to $l<n$ dimensions via the projection $B(v_1,\ldots,v_l)$ with $v_i$ as the $i^{th}$ eigenvector of $A$.  In addition to dimensionality reduction, PCA can be used for visualization, clustering, and other tasks for data analysis.  Such tasks occur in various fields, e.g., genetics \cite{Alt2000,Pri2006}, face recognition \cite{Han96,Zha97}, and signal processing \cite{Lak2004,Har2009}.

In PCA, the eigenvector is typically dense, i.e., each component of the eigenvector is nonzero, and hence the projected variables are linear functions of all original $n$ variables.  In sparse PCA, we restrict the number of variables used in this linear projection, thereby making it easier to interpret the projections.  The additional $l_0$ constraint however makes problem (\ref{eq:sparse_pca_intro}) a difficult  and mathematically interesting problem which arises in many  scientific and engineering applications where very large-scale data sets must be analyzed and interpreted.

Not surprisingly, the search and development of adequate algorithms for solving problem (\ref{eq:sparse_pca_intro}) have
thus received much attention in the past decade, and this will be discussed below. But first, we want to make clear the main  purpose of this paper. We have three main goals:
\begin{itemize}
\item To develop a novel and very simple approach to the $l_0$-constrained PCA problem (\ref{eq:sparse_pca_intro}) as formulated and without any modifications, i.e., no relaxations or penalization, which is
amenable to dimensions in the hundreds of thousands or even millions.
\item To present a ``Father Algorithm'', which we call ConGradU, based on the well-known first-order conditional
gradient scheme, which is very simple, allows for a rigorous convergence analysis, and provides a family of cheap algorithms well-suited to solving various formulations of sparse PCA.
\item To provide a closure and unification to many seemingly disparate approaches recently proposed, and which will be shown to be particular realizations of ConGradU.
\end{itemize}

Most current approaches to sparse PCA can be categorized as solving one of several modified
optimization problems based on penalization, relaxations, or both, and include:
\begin{enumerate}[label=(\alph*)]
\item $l_1$-constrained PCA: $\max{\{x^TAx : \|x\|_2 \leq 1,\|x\|_1\le \sqrt{k}, x\in\reals^n\}}$,
\item $l_0$-penalized PCA: $\max{\{x^TAx - s\|x\|_0 : \|x\|_2\leq 1, x\in\reals^n\}},$
\item $l_1$-penalized PCA: $\max{\{x^TAx -s\|x\|_1: \|x\|_2 \leq 1, x\in\reals^n\}}$,
\item Approximate $l_0$-penalized PCA: $\max\{x^TAx -s g_p(x): \|x\|_2 \leq 1, x\in\reals^n\}$,\\
where $g_p(x) \simeq \|x\|_0$ and $p$ controls the approximation,
\item Convex relaxations.
\end{enumerate}
These models will be presented in detail in the forthcoming section, except for approach (e) which is not thoroughly discussed in this paper (but see \S \ref{ss:literature} below).

In a nutshell, the original $l_0$-constrained PCA problem (\ref{eq:sparse_pca_intro}) and the corresponding modified problems (a)-(d) above can be written or transformed in such a way that they reduce to maximizing a convex function over some compact set $C\subseteq\reals^n$:
$$
(P)\qquad \max\{ F(x):\; x\in C \}.
$$

When the problem of maximizing a {\em linear} function over the compact set $C$ can be
efficiently computed, or even better, can be obtained {\em analytically}, a very simple and
natural iterative scheme to consider for solving (P) is the so-called conditional gradient
algorithm \cite{Levi1966,dunn79}\footnote{The conditional gradient scheme is also known as the Frank-Wolfe
algorithm \cite{Fran1956}. The latter was devised to minimize quadratic convex functions over a
bounded polyhedron, while the former was extended mainly to solve convex minimization problems,
see Section \ref{s:cond_grad_method} for more precise details and relevant references.}.

To achieve the goals alluded to above, in this paper, all developed algorithms for tackling
problem (P) will be based on the conditional gradient scheme with a unit step size called
ConGradU. At this juncture, it is important to notice that Mangasarian [24] seems to have been
the first work suggesting and analyzing the conditional gradient algorithm with a unit step
size for maximizing a convex function over a polyhedron, in the context of machine
learning problems.

A common and interesting appeal of the resulting algorithms is that they take the shape of a
closed-form iterative scheme, i.e., they can be written as
\[
x^{j+1}=\frac{\mathcal{S}(Ax^j)}{\|\mathcal{S}(Ax^j)\|_2}, \; j=0,1,\ldots
\]
where $\mathcal{S}:\reals^n\rightarrow\reals^n$ is a simple operator that can be either
written in explicit form or efficiently computed\footnote{In fact, when
$\mathcal{S}$ is the identity operator, the scheme is nothing else but the well known power method to compute the first principal eigenvector of the matrix $A$, see, e.g., \cite{Str2005}.}.

All problems addressed here are difficult nonconvex
optimization problems and we make no claims with respect to global
optimality;  after all, these are difficult problems so obtaining cheap solutions must have some cost which here is a certificate of global optimality. Moreover, an important driving point is that there is no reason to discount stationary solutions of nonconvex problems versus globally optimal solutions to convex relaxations. For neither solution do we have a measure of the gap to the optimal solution of problem (\ref{eq:sparse_pca_intro}).  We can empirically demonstrate similar solution quality, while the nonconvex methods are orders of magnitudes cheaper to compute and can be applied to data sets much larger than can be done with any known convex relaxation.  This is in contrast to the well-known sparse recovery problem, for which there is an equivalence between the difficult combinatorial problem and the linear program relaxation when the data matrix satisfies certain conditions \cite{Cand05}. To put in perspective the development and results of this paper, we first discuss some of the relevant literature that has motivated this work.

\subsection{Literature}\label{ss:literature}
The literature on sparse PCA can be divided according to the different modifications discussed
above. Here, we briefly survey these approaches and detail them further in Section \ref{s:sparse_pca}. With respect to the $l_0$-constrained PCA problem (\ref{eq:sparse_pca_intro}), thresholding
\cite{Cadi95} is perhaps the simplest approach, however is known to produce poor results.  Sparse low-rank approximations (SLRA) of \cite{Zha2002} looks for a more general ($uv^T$ rather than $xx^T$) sparse approximation by taking an approximation error level as input and determining the sparsity level $k$ that is required to satisfy the desired error.  Greedy methods \cite{Mogh06b} are also computationally expensive due to a maximum eigenvalue computation at each iteration, however an approximate greedy approach \cite{dasp2008} offers a much cheaper way to derive an entire path of solutions (for each $k=0\ldots n$) which
often suffice. These approaches are heuristics. The only globally optimal approach to this formulation
is an exact search method \cite{Mogh06b} that is applicable only to extremely small problems.

The $l_1$-constrained PCA problem is a relaxation, or more precisely an upper bound,
to problem (\ref{eq:sparse_pca_intro}). This problem was first considered in \cite{Joll03}
and called SCoTLASS (simplified component technique least absolute shrinkage and selection).
It was motivated by the LASSO (least absolute selection and shrinkage operator) approach used
in statistics \cite{Tibs96} for inducing sparsity in regression.
In \cite{Tren06}, a true penalty function is used to handle the $l_1$ constraint
and the resulting problem is solved as a system of differential equations
(which requires a smooth approximation for the $l_1$ constraint).
While this approach was tested solely on a small 13-dimensional data set,
we mention it as the only true penalty function approach in the sparse PCA literature.
More recently, a computationally cheap approach to l1-constrained PCA that can solve large-scale problems
was given by \cite{Witt2009}. We will show that this scheme is an application of the conditional gradient algorithm.

As already explained in the introduction, we are not interested in
formulations that require expensive computations.
Convex relaxations are either semidefinite-based or
optimize over symmetric $n\times n$ matrices, and are,
indeed, much more computationally expensive than what we would like to consider here.
 Nevertheless, it is important to briefly recall some of these works.
 In \cite{dasp07}, d'Aspremont et al. introduced a convex relaxation to $l_1$-constrained PCA
 using semidefinite programming, however this formulation can only be solved on very small dimensions ($<100$).
 This was the first approach with convex optimization to any sparse PCA modification and motivated a
 convex relaxation for $l_1$-penalized PCA for which better algorithms were given.
 Another approach for $l_1$-constrained PCA in \cite{Luss2011} solves a different convex
 relaxation over $\symm^n$ based on a variational represention of the $l_1$ ball (
 this relaxation turned out to be the dual of the semidefinite relaxation in \cite{dasp07}).
 This was the first convex relaxation for $l_1$-constrained PCA amenable to a medium number of dimensions (1000-2000).

We next turn to penalized sparse PCA where the sparsity-inducing term appears in the objective
function. Several approaches are known for $l_0$-penalized PCA.  The heuristic given in
\cite{Zha2002} is extended in \cite{Zha2004} to this penalized version of PCA.  In
\cite{dasp2008}, the problem is reformulated as an equivalent convex maximization problem to
which a semidefinite relaxation is applied, however, as mentioned above, solving such problems
is too computationally expensive.  More recently, \cite{Jour2010} derived the same convex
maximization problem (in a different manner) and proposed a first-order gradient-based
algorithm which is identical to ConGradU. Another convex maximization representation was very
recently presented in \cite{Srip2010},
whereby a specific parameterized concave approximation
is used to replace the $l_0$ term, and the resulting problem was solved by an iterative scheme
called the minorization-maximization technique, which is, in fact, a specific instance of
ConGradU.

The $l_1$-penalized PCA problem can be solved by a convex relaxation as
shown in \cite{dasp07}, but is amenable to only a medium number of dimensions.
In \cite{Zou06}, Zhou et al. considered a reformulation of PCA as a two variable
regression problem to which an $l_1$ penalty is added for one of the variables.
While their approach is amenable to larger problems, it is not exactly $l_1$-penalized PCA
and is still more computationally expensive (cf. Section \ref{ss:l1penalized_pca})
than other approaches we will discuss.  Recently, \cite{Jour2010} reformulated
$l_1$-penalized PCA as an equivalent convex maximization problem (as with their $l_0$-penalized approach)
that is also solved by the conditional gradient algorithm.

As discussed above, the conditional gradient algorithm has previously
been applied to sparse PCA in various forms, so we now put the contributions
of the above works into perspective.  The works of \cite{Witt2009,Shen2008}
detail their iterative schemes (without any general algorithm) specifically for sparse PCA.
 \cite{Srip2010} details a general algorithm, similar to ConGradU, but meant for maximizing the
 difference of convex functions over a convex set.
 %for which convergence results are limited.
 While maximizing a convex function is a subset of this algorithm,
 ConGradU, as detailed in Section \ref{s:cond_grad_method}, allows for nonconvex sets.
 \cite{Jour2010} is the only previously known work that
 details a general algorithm for maximizing a convex function over a compact (possibly nonconvex) set.
 Indeed, the first-order algorithm proposed in \cite{Jour2010} (labeled Algorithm 1 therein)
 is identical to what we term the ConGradU algorithm,  however it is
 not recognized as the conditional gradient algorithm.
 As noticed in \cite{Jour2010}, both the $l_0$ and $l_1$-penalized
 PCA algorithms they have proposed were earlier stated in \cite{Shen2008},
 subject to slight modifications, who look for a general rank-one approximation
 (i.e., $uv^T$ rather than $xx^T$); however, no convergence results were stated in \cite{Shen2008}.
 %There is an overlap between our convergence results and those of \cite{Jour2010},
 %which is detailed in Section \ref{ss:cond_grad_method_convergence}.

\subsection{Outline}
We provide computationally simple approaches to both constrained and penalized
versions of sparse PCA. It is important to recognize that all algorithms here are schemes for nonconvex problems; we pay the
price of no global optimality criterion and gain in amenability to problem sizes that convex
relaxations cannot handle.

In Section \ref{s:formulations},  we define the problems of interest and some of their
properties. Section \ref{s:cond_grad_method} recalls some basic optimality results for
maximizing a convex function over a compact set. We then detail ConGradU, a specific
conditional gradient scheme with unit step size, and establish its convergence properties.
Section \ref{s:props} provides a mathematical toolbox proving a series of propositions that
are used to develop the known cheap algorithms mentioned above, as well as for deriving new
schemes.  These propositions are simple and easy to prove so we believe it benefits the reader
to go through these tools first.

Section \ref{s:sparse_pca} then details the
algorithms for all versions of sparse PCA. We start with a simple algorithm for the true
$l_0$-constrained PCA problem (\ref{eq:sparse_pca_intro}). To the author's knowledge, this is the first
available scheme that directly approaches this problem, is amenable to large-scale problems
and proven to converge to a stationary point of problem (\ref{eq:sparse_pca_intro}).  While the $l_0$ constraint is a difficult nonsmooth and nonconvex constraint, we need not look
for ways around this constraint, and rather we approach the given problem as is.  The basis
for our approach is the simple, yet surprising, result (cf. Section \ref{s:props}) that while
maximizing a quadratic function over the $l_2$ unit ball with an $l_0$ constraint is a
difficult problem, maximizing a linear function over the same nonconvex set is simple and can
be solved in $O(n)$ time. An important aspect of the new $l_0$-constrained PCA algorithm is
that no parameters need be tuned in order to obtain a stationary point that has the exact
desired sparsity\footnote{All other algorithm based on modifications can be used to obtain
a desired sparsity as well, however parameters must be tuned accordingly.}. The second main
part of Section \ref{s:sparse_pca} focuses on all aforementioned iterative schemes which have
been proposed in the literature. Building on the results of Section \ref{s:props}, we show that all these schemes can directly be obtained as a particular realization of ConGradU, or of some variant of it, thus providing a unifying framework to various seemingly different algorithmic approaches.

Section \ref{s:experiments} provides experimental results and demonstrates the efficiency of many of the methods we have reviewed on large-scale problems. We show that they all give comparable solutions, i.e., very similar $k$-sparse solutions, with the advantage of $l_0$-constrained PCA being that the $k$-sparse solution is directly obtained at a lower computational cost. Section \ref{s:extensions} ends with concluding remarks and briefly shows how to use the same tools to develop simple algorithms for related sparsity constrained problems.

\subsection{Notation} We write $\symm^n$ ($\symm^n_+,\symm^n_{++}$) to denote the set of symmetric
(positive-semidefinite, positive-definite) matrices of size $n$ and $\reals^n$
($\reals^n_+,\reals^n_{++}$) to denote the set of (nonnegative, strictly positive) real
vectors of size $n$. The vector $e$ is the $n$-vector of ones.  Given a vector $x\in\reals^n$, $\|x\|_2=(\sum_i{x_i^2})^{\frac{1}{2}}$ defines the $l_2$ norm, $\|x\|_0$ defines the cardinality of $x$, i.e., the number of nonzero entries of $x$ and usually called here the $l_0$ norm\footnote{We note an abuse of terminology because $\|x\|_0$ is not a true norm since it is not positively homogenous.}, and $\|x\|_\infty=\max{(|x_1|,\ldots,|x_n|)}$.  Given a matrix $X\in\symm^n$, $\|X\|_\infty=\max_{i,j}{X_{i,j}}$ and $\|X\|_F=(\sum_{i,j}{X_{i,j}^2})^{\frac{1}{2}}$. For a vector $x\in\reals^n$, $|x|$ denotes the
vector with $i^{th}$ entry $|x_i|$, $\mbox{sgn}(x)$ denotes the vector with $i^{th}$ entry
-1,0,1  if $x_i<0,x_i=0,x_i>0$, and $x_+$ denotes the vector with $i^{th}$ entry
$\max{(x_i,0)}$. For a vector $x\in\reals^n$, $\mbox{diag}(x)$ denotes the diagonal matrix
with $x$ on its diagonal. For any nonzero integer $n$, denote the set $\{1, \ldots ,n\}$ as $[n]$.  Let $I_n$ denote
the identity matrix in dimension $n$. Let $\mathcal{C}^1$ denote the space of once
continuously differentiable functions on $\reals^n$. Given an optimization problem (P), we use
$\argmax (P)$ to denote its optimal solutions set.

\section{Problem Formulations}\label{s:formulations}
This section describes the relationship between $l_0$-constrained PCA and the various modified sparse PCA problems that are discussed throughout the paper, as well as certain properties.
\subsection{The Original Optimization Model}
We start with some useful and elementary properties of the  $l_0$-constrained PCA problem.\\

\noindent {\bf The $l_0$-constrained PCA Problem}\\
Given a symmetric matrix $A\in S^n$ and sparsity level $k \in [1,n]$, the main
problem of interest is to solve the $l_0$-constrained PCA problem (i.e., the sparse eigenvalue problem):
\begin{equation} \label{eq:sparse_pca}
(E)\qquad\qquad \max\{x^TAx : \|x\|_2=1, \|x\|_0\leq k,x\in\reals^n\}.
%\equiv\max\{q(x):x\in C\}.
\end{equation}

In most applications, $A$ is the covariance matrix of some data matrix $B\in\reals^{m\times
n}$ such that $A=B^TB$ and hence is positive semidefinite. The latter fact will be exploited
in reformulations of the problem described below. In fact, as is very well-known, since problem
(E) is constrained by the unit sphere, without loss of generality (see e.g., \cite{dasp2008,
dasp07,Jour2010}) we can always assume that $A$ is positive semidefinite since we clearly have
$$\max_{x\in\reals^n}\{x^TA_\sigma x: \|x\|_2=1, \|x\|_0\leq k\} =
\max_{x\in\reals^n}\{x^TAx: \|x\|_2=1, \|x\|_0\leq k\} +\sigma,$$ with $\sigma
>0$ such that $A_\sigma:=A +\sigma I_n \in S^n_{++},$ i.e., with respect
to optimal objective values we are solving the same problem.

% by simply the
%objective function $x^TA_\sigma x$
%However, in the case $A$ is not positive
 %semidefinite, this difficulty can be easily overcome as follows. Let  $\sigma >0$ be such
 %that  $A_\sigma:=A +\sigma I_n \in S^n_{++}$ and consider the objective function $x^TA_\sigma x$. Then clearly,
 %For any $\sigma > -\lambda_{\min}(A)$, define $A_\sigma:=A +\sigma I_n$ and consider
 %the objective function $q_\sigma(x):=x^TA_\sigma x$. Then clearly, $A_\sigma \in S^n_{++}$ and
%$$\max_{x\in\reals^n}\{x^TA_\sigma x: \|x\|_2=1, \|x\|_0\leq k\} =
%\max_{x\in\reals^n}\{x^TAx: \|x\|_2=1, \|x\|_0\leq k\} +\sigma,$$
%i.e., with respect to optimal objective values we are solving the same problem (this was observed in \cite{Jour2010}).
%For example, taking $\sigma$ to be larger than the result of running the power iteration scheme on $A$ would suffice,
%since then $\sigma>\max{(|\lambda_{\max}(A)|,|\lambda_{\min}(A)|)}$, where $\lambda_{\max}(A)$ and $\lambda_{\min}(A)$
%are the maximum and minimum eigenvalues of $A$.
The next
result furthermore states that we can relax the sphere constraint to its convex counterpart,
the unit ball, namely we consider the problem
$$ (E_\sigma) \qquad\qquad \max\{x^TA_\sigma x:  \|x\|_2\leq 1, \|x\|_0\leq k, x\in\reals^n\} $$
and also show that problems $(E)$ and $(E_\sigma)$ are equivalent and admit the same set of
optimal solutions. The simple proof is omitted.

\begin{lemma}\label{equiv-prob} Fix any $\sigma >0$ such that $A_\sigma \in S^n_{++}$. Then,\\
(a) $\max \{x^TA_\sigma x: \|x\|_2 \leq 1, \|x\|_0\leq k\}= \max\{x^TAx : \|x\|_2=1, \|x\|_0\leq k\}
+ \sigma.$\\
(b) $\argmax (E)= \argmax (E_\sigma)$.
\end{lemma}

Needless to say  that both problems $(E)$ and $(E_\sigma)$ remain  hard nonconvex problems as they consist of maximizing a convex  (strongly convex) function over a compact set, and
clearly the possibility of using the  convex relaxation $\{x: \|x\|_2\leq 1\}$ instead of the
nonconvex unit sphere constraint does not change the situation. However, it is useful to know
that either one of these constraints can be used when tackling the problem, and this will
be done throughout the rest of the paper without any further mentioning.

Throughout the paper, $(E)$ will be referred to as the $l_0$-constrained PCA problem, and $A$ always denotes a symmetric matrix while $A_\sigma$ will denote a symmetric positive definite matrix.  We now consider the other formulations that will be analyzed.

\subsection{Modified Optimization Models}
As mentioned in the introduction, most approaches to sparse PCA solve one of the following
modified problems.  The first variation is a relaxation based on the relation
\BEQ \label{eq:card_ineq}
\|x\|_1\le\sqrt{\|x\|_0}\|x\|_2 \qquad \forall
x\in\reals^n
\EEQ
which follows from the Cauchy-Schwarz inequality.  The hard $l_0$ constraint in problem (\ref{eq:sparse_pca}) is replaced by an $l_1$ constraint, resulting in\\

\noindent{\bf The $l_1$-constrained PCA Problem}\\
\begin{equation}
\label{eq:sparse_pca_l1constrained}
 \max{\{x^TAx : \|x\|_2 \leq 1,\|x\|_1\le \sqrt{k}, x\in\reals^n\}},
\end{equation}
which thanks to inequality (\ref{eq:card_ineq}) is an upper bound to the original $l_0$-constrained PCA problem.  Two other variations are based on penalizations of the $l_0$ and $l_1$ constraints of the above formulations.  Note that the penalized terminology is different from the usual one used in optimization, and here is used to mention that it rather optimizes a tradeoff between how good and how sparse the approximation is.  We first penalize the $l_0$ constraint in problem (\ref{eq:sparse_pca}), resulting in\\

\noindent{\bf The $l_0$-penalized PCA Problem}\\
\begin{equation} \label{eq:sparse_pca_l0penalized}
\max{\{x^TAx - s\|x\|_0 : \|x\|_2\leq 1, x\in\reals^n\}},
\end{equation}
where $s>0$ is a parameter that must be tuned to achieve the truly desired sparsity level at which $\|x\|_0=k$.  However, to avoid the trivial optimal solution
$x^*(s)\equiv 0$, the parameter $s$ must be restricted. First recall the well-known norm relations
\begin{equation}\label{sec2:nrelat}
\|x\|_\infty \leq \|x\|_2 \leq \|x\|_1, \; \forall x \in \reals^n.
\end{equation}
Using the H\"{o}lder inequality\footnote{For any $u,v \in \reals^n$, the H\"{o}lder inequality states
that $|\langle u,v\rangle| \leq \|u\|_p \|v\|_q$, where $p+q=pq, p\geq 1.$}, combined with inequalities (\ref{eq:card_ineq}) and (\ref{sec2:nrelat}), it follows that
$$
x^TAx - s\|x\|_0 \leq \|A\|_\infty \|x\|_1^2-s \|x\|_0 \leq (\|A\|_\infty -s)\|x^*\|_0,
$$
for all $x$ feasible for problem (\ref{eq:sparse_pca_l0penalized}). Thus, to avoid the trivial solution, it is assumed that $s\in (0,\|A\|_\infty)$.  Note that while $s\ge\|A\|_\infty$ necessarily implies the trivial solution, taking $s<\|A\|_\infty$ does not guarantee we avoid it, but only gives a particular bound.

Likewise, a penalized version of the $l_1$-constrained PCA problem
(\ref{eq:sparse_pca_l1constrained}) yields\\

\noindent{\bf The $l_1$-penalized PCA Problem}\\
\begin{equation}
\label{eq:sparse_pca_l1penalized}
\max{\{x^TAx -s\|x\|_1: \|x\|_2 \leq 1, x\in\reals^n\}}.
\end{equation}
Note that, as in problem (\ref{eq:sparse_pca_l0penalized}), we
need to restrict the value of the parameter $s$ in order to avoid the trivial solution.
Again, using the H\"{o}lder inequality and (\ref{sec2:nrelat}), it is easy to see that
$$x^TAx -s\|x\|_1 \leq (\|A\|_F -s)\|x^*\|_1,$$ for all $x$ feasible for problem (\ref{eq:sparse_pca_l1penalized}), and hence it is assumed that $s\in (0, \|A\|_F)$.

Again, note that the formulated penalized/relaxed problems remain hard nonconvex maximization problems despite the convexity of  their constraints.  In fact, $l_0$-penalized PCA and the $l_1$-penalized version share an additional difficulty in that their objectives are neither concave nor convex. In Sections \ref{ss:l0penalized_pca} and \ref{ss:l1penalized_pca}, we show how this difficulty is overcome.\\

\noindent{\bf The Approximate $l_0$-penalized PCA Problem}\\
The last approach for solving problem (\ref{eq:sparse_pca}) involves approximating the $l_0$ norm in the objective of  $l_0$-penalized PCA. The idea of
approximating the $l_0$ norm by some nicer continuous functions naturally emerged from very
well-known mathematical approximations of the step and sign functions (see, e.g., \cite{Brac2000}).  Indeed, it is easy to see that for any $x\in \reals^n$, one can write
$$
\|x\|_0 =\sum_{i=1}^{n} \sgn(|x_i|).
$$
Thus, formally, we want to replace the problematic expression $\sgn(|t|)$ by some nicer function
and consider an approximation of the form
\[
\|x\|_0=\lim_{p\rightarrow0}{\displaystyle\sum_{i=1}^n{\varphi_p(|x_i|)}},
\]
where $\varphi_p:\reals_+\rightarrow\reals_+$ is an appropriately chosen smooth function.  Here, we consider the class of smooth concave functions which are monotone increasing and normalized such that $\varphi_p(0)=0, \varphi_p'(0)>0$.

This suggests to approximate problem (\ref{eq:sparse_pca_l0penalized}) by considering
for $p,s>0$ the following approximate maximization problem:
\begin{equation} \label{eq:approx}
\max\{x^TAx - s \sum_{i=1}^n \varphi_p(|x_i|) : \|x\|_2\leq1, x\in\reals^n\}.
\end{equation}

Approximations of the $l_0$ norm have been considered in various applied contexts, for instance in the machine learning literature, see, e.g., \cite{Mang1996,West2003}.

There exist several possible choices for the function $\varphi_p(\cdot)$ that approximate the
step function, and for details the reader is referred  to the classic book by Bracewell \cite[Chapter
4]{Brac2000}.  For illustration, here we give the following examples as well as their graphical representation in Figure \ref{fig:l0_concave_approximations} (left).

\begin{example}\label{exphi}Concave functions $\varphi_p(\cdot),p>0$
\begin{enumerate}[label=(\alph*)]
\item $\varphi_p(t)=(2/\pi)\tan^{-1}(t/p)$,
\item $\varphi_p(t)= \log(1+t/p)/\log(1+1/p)$,
\item $\varphi_p(t)= (1 + p/t)^{-1}$,
\item $\varphi_p(t)= 1-e^{-t/p}.$
\end{enumerate}
\end{example}

\begin{figure}[h!] \begin{center}
  \begin{tabular} {cc}
     \psfrag{title}[t][t]{\small{Approximations $\varphi_p(t)$ with $p=.05$}}
     \psfrag{y}[b]{\small{$\varphi_p(t)$}}
     \psfrag{x}[t]{\small{t}}     \includegraphics[width=0.49\textwidth]{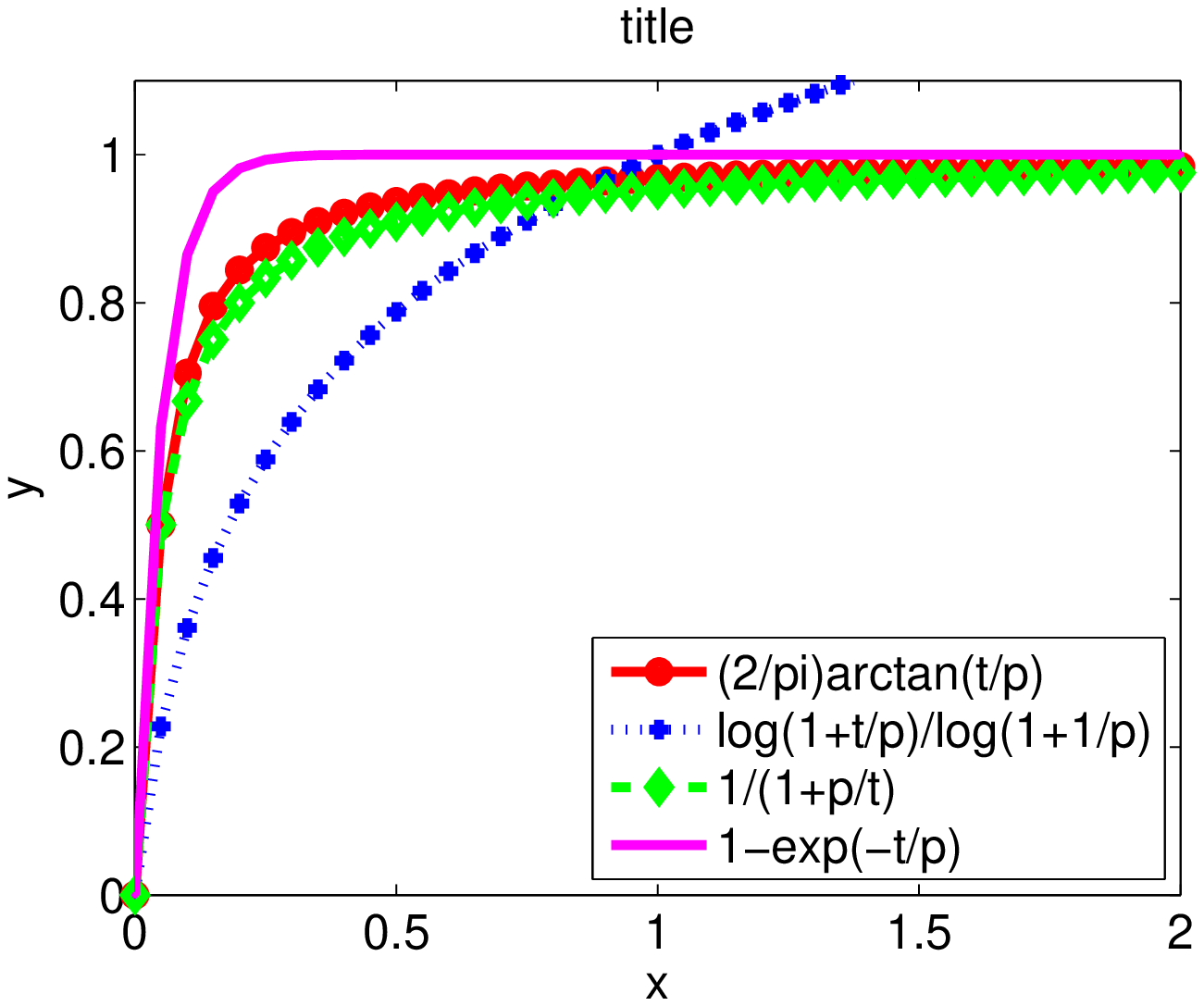}&
     \psfrag{title}[t][t]{\small{Approximation as $p\rightarrow0$}}
     \psfrag{y}[b]{\small{$1-e^{-t/p}$}}
     \psfrag{x}[t]{\small{t}}
\includegraphics[width=0.49\textwidth]{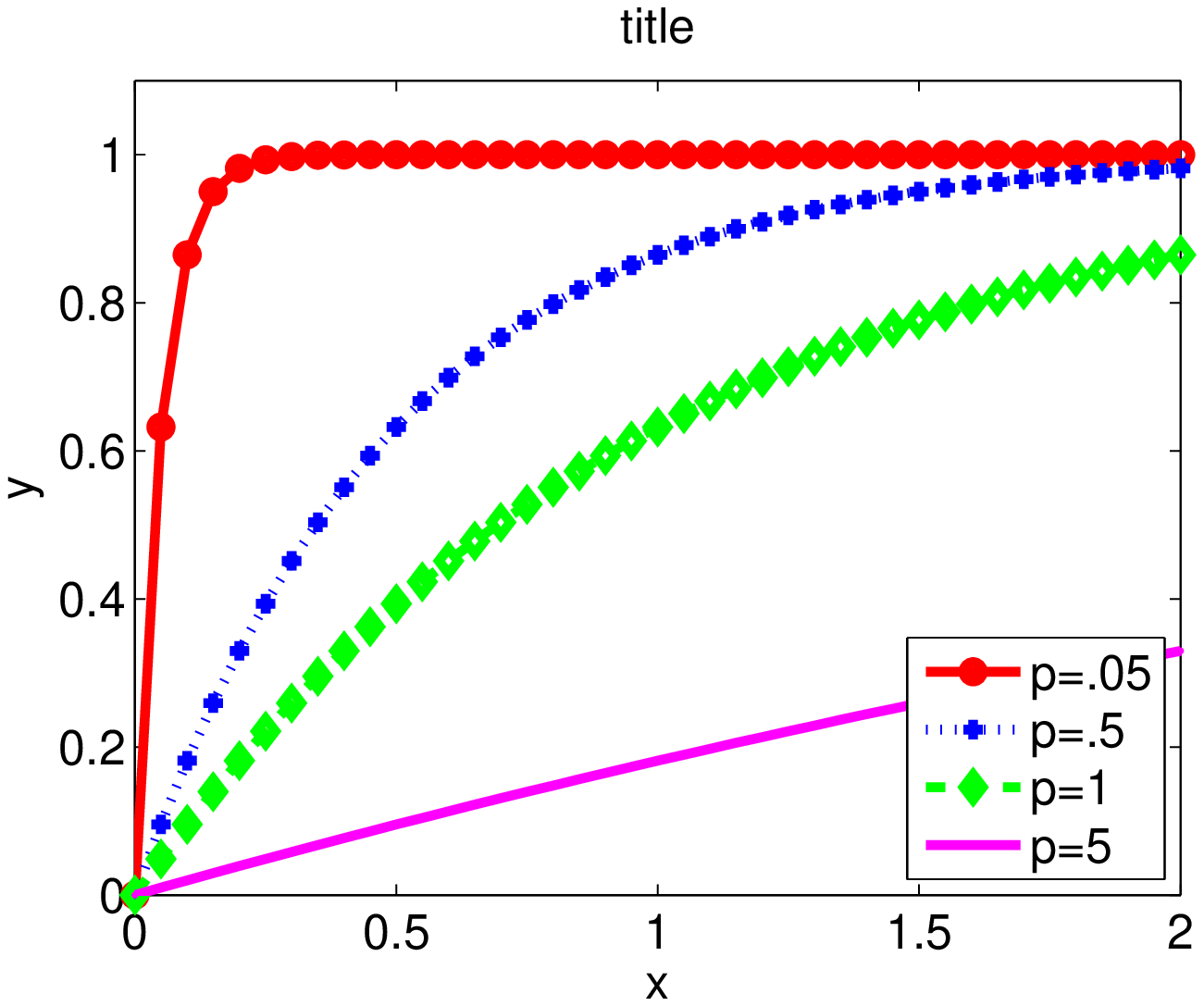}
  \end{tabular}
\caption{The left plot shows four concave functions $\varphi_p(t)$ that can be used to approximate $\|x\|_0\approx\sum_{i=1}^n{\varphi_p(|x_i|)}$ for fixed $p=.05$.  The right plot shows how the concave approximation $1-e^{-t/p}$ converges to the indicator function as $p\rightarrow0$.}
\label{fig:l0_concave_approximations}\end{center}\end{figure}

The last example (d) was successfully used in the context of machine learning by Mangasarian \cite{Mang1996}, and gives
\begin{equation}
\varphi_p(|t|)=1-e^{-p|t|}=\begin{cases}
0 & \text{if $t=0$}\\
>0 & \text{if $t\ne0$.}
\end{cases}
\end{equation}
A nice feature of this example is that it also lower bounds the $l_0$ norm, namely, we have
$$ \sum_{i=1}^n{\varphi_p(|x_i|)}\leq\|x\|_0, \quad \forall x\in\reals^n.$$
See Figure \ref{fig:l0_concave_approximations} (right) for its behavior for various values of $p$.

All problems listed in this section, as well as several other equivalent reformulations that will be derived in Section \ref{s:sparse_pca}, will be solved by ConGradU, a conditional gradient algorithm with unit step size for maximizing a convex function over a compact set which is described next.

\section{The Conditional Gradient Algorithm} \label{s:cond_grad_method}

\subsection{Background} The conditional gradient algorithm, is a well-known and simple gradient algorithm,
see, e.g., \cite{Levi1966,dunn79} and the book \cite{Bert1999} for a general overview of this method as well as more references. Note that this algorithm dates back to
1956, and is also known as the Frank-Wolfe algorithm \cite{Fran1956} that was originally proposed for solving linearly constrained quadratic programs.

The conditional gradient algorithm is presented here for maximization problems because of our interest in the sparse PCA problem.  We first recall the conditional gradient algorithm for maximizing a continuously differentiable function $F:\reals^n\to \reals$ over a nonempty compact convex set $C \subset
\reals^n$:
\begin{equation} \label{eq:fw}
\max{\{F(x) : x\in C\}}.
\end{equation}
The conditional gradient algorithm generates a sequence $\{x^j\}$ via the iteration:
$$
x^0 \in C, \; x^{j+1}=x^j+\alpha^j (p^j -x^j), \; j=0,1,\ldots $$ where
\begin{equation}\label{cgp}
p^j=\argmax{\{\langle x-x^j,\nabla F(x^j)\rangle : x\in C\}},
\end{equation}
and where $\alpha^j \in (0,1]$ is a stepsize that can be determined by the Armijo or limited
maximization rule \cite{Bert1999}. It can be shown that every limit point of the sequence
$\{x^j\}$ generated by the conditional gradient algorithm is a stationary point. Furthermore,  under various additional assumptions on the function $F$ and/or the set $C$ (e.g., strong convexity of the function $F$ and/or of the set $C$), rate of convergence results can be established, see in particular the work of Dunn \cite{dunn80} and references therein.

Clearly, the conditional gradient algorithm becomes attractive and simple when its main computational step (\ref{cgp}) can be performed efficiently, e.g., when $C$ is a bounded polyhedron it reduces to solving a linear program, or even better when it can be solved {\em analytically}.

\subsection{Maximizing a Convex Function via ConGradU} \label{ss:cond_grad_method_convergence}
As we shall see below, all potential reformulations of the sparse PCA problem will lead to maximizing a {\em convex} (possibly nonsmooth) function over a compact (possibly nonconvex) set $C\subset \reals^n$.  It will be shown that, for such a class of problems, the conditional gradient algorithm will reduce to a very simple iterative scheme.  First we recall some basic definitions and properties relevant to maximizing convex functions.

Let $F:\reals^n\to \reals$ be convex, and $C\subset \reals^n$ be a nonempty compact set.
Throughout, we assume that $F$ is not constant on $C$. We denote by $F'(x)$ any subgradient of the convex function $F$ at $x$ which
satisfies
\begin{equation}\label{subgrad}
F(v)-F(x) \geq \langle v-x, F'(x) \rangle, \; \forall v\in \reals^n.
\end{equation}
The set of all subgradients of $F$ at $x$ is the subdifferential of the function $F$ at $x$,
denoted by $\partial F(x)$, i.e.,
$$
\partial F(x)=\{g:\; F(v) \geq F(x) + \langle v-x, g \rangle, \; \forall v\in \reals^n\},
$$
which is a closed convex set. When $F\in\mathcal{C}^1$, the subdifferential reduces to a singleton which is the gradient of $F$, that is $\partial F(x)=\{\nabla F(x)\}$, and
(\ref{subgrad}) is the usual gradient inequality for the convex function $F$. In the following, we use the notation $F'(\cdot)$ to refer to either a gradient or subgradient of $F$; the context will be clear in the relevant situation.

The  first result recalls two useful properties when maximizing a convex function (see \cite[Section 32]{Rock1970}).
\begin{proposition}\label{max-basic} Let $F:\reals^n\to \reals$ be convex,   $S\subset \reals^n$ be
an arbitrary set and let $\mbox{conv}(S)$ denote its convex hull. Then,\\
(a) $\sup \{F(x): x\in \mbox{conv}(S)\}=\sup \{F(x): x\in S \}$, where the first supremum
is attained only if the second is attained.\\
(b) If $S \subset \reals^n$ is closed with nonempty boundary $\mbox{bd}(S)$, and $F$ is bounded
above on $S$, then $\sup \{F(x): x\in S\}=\sup \{F(x): x\in \mbox{bd}(S)\}.$
\end{proposition}

The next result gives a first-order optimality criterion  for maximizing a convex function $F$ over a compact set $C\in\reals^n$. It uses property (a) of Proposition \ref{max-basic} and
follows from \cite[Corollary 32.4.1]{Rock1970}.
\begin{proposition}\label{max-opt}
Let $F:\reals^n \to \reals$ be convex. If $x$ is a local maximum of $F$ over the nonempty
compact set $C$, then
\begin{equation}\label{opt}
(FOC)\qquad \langle v-x, F'(x) \rangle \leq 0, \quad \forall v \in C.
\end{equation}
\end{proposition}

When $F\in\mathcal{C}^1$, a point $x \in C$ satisfying the first-order optimality criteria (FOC)
(\ref{opt}) will be referred as a stationary point, otherwise, when $F$ is convex nonsmooth,
we will say that the point $x\in C$ satisfies (FOC).

We are now ready to state the algorithm. It turns out that when the function $F$ is convex and quadratic, we can also eliminate the need for finding a step size and consider the conditional gradient algorithm with a fixed unit step size and still preserve its convergence properties. This simplified version of the conditional gradient algorithm will be used throughout the paper and referred to as ConGradU.

\begin{algorithm}
\caption{{\bf ConGradU -- Conditional Gradient Algorithm with Unit Step Size}}
\begin{algorithmic} [1]
\REQUIRE {$x^0\in C$} \\
\medskip
\STATE $j \leftarrow 0, $ \WHILE {stopping criteria} \STATE $x^{j+1}\in \argmax\{ \langle
x-x^j, F'(x^j)\rangle : x\in C\}$ \ENDWHILE \RETURN $x^j$
 \end{algorithmic}
 \label{alg:frank_wolfe}
\end{algorithm}

To analyze ConGradU, in what follows, it will be convenient to introduce the quantity
\begin{equation}\label{gap}
\gamma(x) := \max \{\langle v-x, F'(x) \rangle:\; v\in C \}
\end{equation}
for any $x\in \reals^n$. Since $C$ is compact, this quantity is well-defined and admits a global maximizer $$u(x)\in \argmax{\{\langle v-x, F'(x) \rangle:\; v\in C\}},$$
and  thus, we have $\gamma (x) = \langle u(x)-x, F'(x) \rangle$.

In terms of the above defined quantities, we thus have that $x^*$ satisfies (FOC) is
equivalent to saying that $x^*$ is a global maximizer of $\psi(v)=\langle v-x^*, F'(x^*)\rangle$, i.e.,
$$
x^*\in\argmax \{ \langle v-x^*, F'(x^*)\rangle:\; v \in C\}=u(x^*).
$$
Thus the ConGradU algorithm is nothing else but a fixed point scheme for the map $u(\cdot)$, and simply reads as
$$x^0\in C, \; x^{j+1}=u(x^j),\quad j=0,1, \ldots .$$

\begin{lemma}\label{basic} Let $F:\reals^n\to \reals$  be convex, $C\subset \reals^n$ be nonempty and compact,
 and let $\gamma (x)$ be given by (\ref{gap}). Then,\\
(i) $\;\gamma (x) \geq 0\;$ for all $x\in C$.\\
(ii) For any $v \in C$ and any $x\in \reals^n$,
$$
F(u(x))-F(x) \geq \gamma (x) \geq \langle v-x, F'(x) \rangle.
$$
\end{lemma}
\begin{proof} The proof of (i) and the right inequality of (ii) follows immediately from the
definition of $\gamma (x)$, while the left inequality in (ii) follows from the subgradient inequality for the convex function $F$:
\[ F(u(x)) - F(x) \geq \langle u(x)-x, F'(x)\rangle= \gamma (x).\]  \end{proof}

We are ready to state the convergence properties of ConGradU.
\begin{theorem} \label{thm:fw_stationarity_convex_obj}
Let $F:\reals^n \to \reals$ be a  convex function and let $C\subset \reals^n$ be nonempty
compact. Let  $\{x^j\}$ be the sequence generated by the algorithm ConGradU. Then the
following statements hold:\\
(a) the sequence of function values $F(x^j)$ is monotonically increasing and
$$
\lim_{j\to \infty} \gamma (x^j)=0.
$$
\noindent(b) If for some $j$ the iterate $x^j$ satisfies $\gamma (x^j)=0$, then the algorithm stops with $x^j$ satisfying (FOC). Otherwise the algorithm generates an infinite sequence $\{x^j\}$
with strictly increasing function values $\{F(x^j)\}$.\\
\noindent(c) Moreover, if $F$ is
continuously differentiable, then every limit point of the sequence $\{x^j \}$  converges to a
stationary point.
\end{theorem}
\begin{proof} By definition of the iteration of ConGradU, using our notations, the sequence $\{x^j\}$ is well defined
via $x^{j+1}\in u(x^j), \forall j=0, 1\ldots .$ Invoking Lemma \ref{basic}, we obtain
$$
0\leq \gamma (x^j) \leq F(x^{j+1})-F(x^j),\; \forall j=0, 1, \ldots,
$$
showing that the sequence $\{F(x^j)\}$ is monotone increasing. Summing the inequality above
for $j=0, \ldots, N-1$, we have
$$
\sum_{j=0}^{N-1} \gamma (x^j) \leq F(x^N)-F(x^0).
$$
Since $C$ is compact and $F(\cdot)$ is continuous, we also have $F(x^N) \leq \max \{F(x):\; x
\in C\}:=F_*$, and thus it follows from the above inequality that $\sum_{j=0}^{N-1} \gamma
(x^j) \leq F_*-F(x^0)$, and hence the nonnegative series $\sum_{j=0}^\infty \gamma(x^j)$ is
convergent so that $\gamma(x^j)$ converges to 0. Now, since $\gamma (x^j) \geq 0$ for all
$j=0, \ldots$, then if for some $j$ the iterate  $x^j$ is such that $\gamma (x^j)=0$ then the
procedure stops at iteration $j$ with $x^j$ satisfying (FOC). Otherwise, $\gamma (x^j) >0$ and
the iteration generates an infinite sequence with $F(x^{j+1}) > F(x^j)$. In the latter case,
assuming now that $F\in C^1$, we will now prove the last statement of the theorem. Since $C$
is compact, the sequence $\{x^j\}\subset C$ is bounded. Passing to subsequences if necessary,
for any limit point $x^\infty$ of $\{x^j\}$ we thus have $x^j\to x^\infty$. Without loss of generality we
let $u(x^j) \to {\bar u}$. Using the facts
$$ \gamma(x^j)=\langle u(x^j)-x^j, F'(x^j)\rangle\quad\mbox{and}\quad\langle v-x^j,F'(x^j)\rangle \leq
\langle u(x^j)-x^j, F'(x^j) \rangle, \; \forall v \in C,$$
and since $\gamma(x^j)\to 0$ and $F\in\mathcal{C}^1$, passing to the limit over appropriate subsequences in the above relations, it follows
that $\langle {\bar u} -x^\infty, F'(x^\infty)\rangle=0$ and $\langle
v-x^\infty,F'(x^\infty)\rangle \leq \langle {\bar u} -x^\infty, F'(x^\infty) \rangle, \;
\forall v \in C$, and hence $\langle v-x^\infty, F'(x^\infty)\rangle \leq 0, \forall v\in C$,
showing that $x^\infty$ is stationary.
\end{proof}

\begin{remark}\label{rem:journ} (a) For the case of convex
$F$ and bounded polyhedron $C$, as noted in the introduction, Mangasarian \cite{Mang1996}
seems to have been the first work suggesting the possibility of using a unitary step size in
the conditional gradient scheme and proved that the algorithm generates a finite sequence
(thanks to the polyhedrality of $C$) that terminates at a stationary point.

(b) Very recently, the use of a unitary stepsize in the conditional gradient scheme was
rediscovered in \cite{Jour2010} with $C$ being an arbitrary compact set, which is identical to
Algorithm \ref{alg:frank_wolfe}.

(c) The proof of Theorem \ref{thm:fw_stationarity_convex_obj} is patterned after the one given
in \cite{Mang1996}. Note that part (a) also follows from \cite{Jour2010} as well. Furthermore,
under stronger assumptions on $F$ and $C$, \cite[Theorem 4]{Jour2010} also established a
stepsize convergence rate giving an upper estimate on the number of iterations the algorithm
takes to produce a step of small size.

%a stepsize convergence rate for $\gamma_k=\max_{0\leq
%i\leq k}{\gamma(x^i)}$, and under stronger assumptions on $F$ and $C$, a convergence rate for
%$\min_{0\leq i\leq k}{\|x^{i+1}-x^{i}\|_2}$;
 %a convergence rate for $F(x^*)-F(x^k)$ with nonconvex sets remains unknown.

 (d) Finally, as kindly pointed out by a referee, the proof of convergence could also probably be derived
 by using the general approach developed in the classical monograph of Zangwill \cite{Zang69}.
 %to the authors' knowledge, known stationarity results for
 %the conditional gradient algorithm always assume $C$ to be convex,
 %while, part (c) in Theorem \ref{thm:fw_stationarity_convex_obj} only assumes compactness.
\end{remark}

We end this section with a particular realization of ConGradU for an interesting class of problems given as
$$
(G) \qquad \max_x{\{f(x)+g(|x|) : x\in C\}}
$$
where
$$
\begin{array}{ll}
f:\reals^n\rightarrow\reals&\mbox{is convex},\\
g:\reals^n_+\rightarrow\reals&\mbox{is convex differentiable and monotone decreasing}\footnotemark,
\\
C\subseteq\reals^n&\mbox{is a compact set.}
\end{array}
$$
\footnotetext{By monotone decreasing, we mean componentwise, i.e., each component of $g'(\cdot)$ the gradient of $g$ is such that $g'_{i}(\cdot)_i<0\quad\forall i\in[n].$}

Our interest for this form is mainly driven by and is particularly useful for handling the
case of the approximate $l_0$-penalized problem (cf. Section \ref{ss:l0penalized_pca}), which
precisely has the form of this optimization model with adequate choice of the kernel
$\varphi_p$ that can be used to approximate the $l_0$ norm (cf. Section \ref{s:formulations}).
It will also allow for developing a novel simple scheme for the $l_1$-penalized problem (cf. Section \ref{ss:l1penalized_pca}).

Note that under the stated assumptions for $g$, the composition $g(|x|)$ is not necessarily convex
and thus ConGradU cannot be  applied to problem (G). However, thanks to
the componentwise monotonicity of $g(\cdot)$, it is easy to see
that problem (G) can be recast as an equivalent problem with additional constraints and
variables as follows:
$$
(GG)\qquad \max_{x,y}\{ f(x)+g(y) :  |x|\leq y, x\in C \}.
$$
Note that, without loss of generality, an additional upper bound can be imposed on $y$ in order to enforce compactness of the feasible region of problem (GG) (e.g.,  by setting the upper bound on $y$ as $y_c:=\argmax\{|x|: x \in C\}$), but this need not be computed in order to establish our following result.  Clearly, the new objective of (GG) is now convex in $(x,y)$ and thus we can apply ConGradU. We show below that the main iteration in that case leads to an attractive weighted
$l_1$-norm maximization problem, which, in turn, as shown in Section \ref{s:props} can be
solved in closed form for the cases of interest, namely when $C$ is the compact set described
by the unit sphere or unit ball.

\begin{proposition} \label{prop:nonsmooth_condtional_gradient}
The algorithm ConGradU applied to problem (GG) generates a sequence $\{x^j\}$ by solving the weighted $l_1$-norm maximization problem
\begin{equation}\label{eq:fw_direction_search_modified}
x^0 \in C, \; x^{j+1}=\argmax \{\langle a^j,x\rangle -\sum_i{w^j_i|x_i|} : x\in C \}, j=0,
\ldots ,
\end{equation}
where $w^j=-g'(|x^j|)>0$ and $a^j=f'(x^j) \in \reals^n$.
\end{proposition}
\begin{proof}
Applying ConGradU to the convex function $H(x,y):=f(x)+g(y)$, with $y^0=|x^0|, x^0\in C$
we obtain,
\begin{eqnarray*} \label{eq:fw_direction_search_smooth}
(x^{j+1},y^{j+1})&=&\argmax_{x,y}\{\langle x-x^j, f'(x^j)\rangle+\langle y-y^j,
g'(y^j)\rangle : x\in C, |x|\leq y\}\\
 &=& \argmax_{x \in C} \left\{ \langle x-x^j, f'(x^j)\rangle + \max_{y}\{ \langle y-y^j, g'(y^j)\rangle
 : |x|\leq y\} \right\}\\
&=& \argmax_{x \in C} \left\{ \langle x-x^j, f'(x^j)\rangle +  \langle |x|-y^j, g'(y^j)\rangle
\right\},
\end{eqnarray*}
where the last max computation with respect to $y$ uses the fact that $g'$ is monotone
decreasing. It is also clear that, given the initialization $y^0=|x^0|$ and $y^{j+1}=\argmax{\{\langle y,g'(y^j)\rangle:|x^{j+1}|\leq y\}}$, we have $y^j=|x^j|$ for all $j=0,1,\ldots$ Omitting constant terms, the last iteration
can be simply rewritten as
$$x^{j+1}=\argmax_{x \in C}\{ \langle x, f'(x^j)\rangle +  \langle |x|,
g'(|x^j|) \rangle \},$$
which with $w^j:=-g'(|x^j|)>0$ proves the desired result stated in (\ref{eq:fw_direction_search_modified}). \end{proof}

%For the case $C=\{x\in\reals^n:\|x\|_2=1\}$, we show in Section \ref{s:props}
%that the problem (\ref{eq:fw_direction_search_modified}) can be solved in closed form.

\section{A Simple Toolbox for Building Simple Algorithms} \label{s:props}
The algorithms discussed in this paper are based on the fact that the $l_0$-constrained PCA
problem as well as many of the penalized and constrained $l_0$ and $l_1$ optimization problems
that are solved by the proposed conditional gradient algorithm have iterations that either have
closed form solutions or are easily solvable. Specifically, the main step of ConGradU
maximizes a linear function over a compact set, and the following lemmas and propositions show
that for certain compact sets (e.g $\{x\in\reals^n:\|x\|_2=1,\|x\|_0\le k\}$ and
$\{x\in\reals^n:\|x\|_2=1,\|x\|_1\le k\}$), these subproblems are easy to solve.

In addition, propositions are given that will be used to reformulate
(in Section \ref{s:sparse_pca}) problems such as $l_0$ and $l_1$-penalized PCA,
 which have neither convex nor concave objectives, to problems that maximize a convex objective function.
 The propositions for maximizing linear functions over compact sets can then be used in ConGradU
 as applied to the reformulated $l_0$ and $l_1$-penalized PCA problems.
 Combined with the conditional gradient algorithm discussed above,
 these propositions are the only tools required throughout the remainder of the paper.

We start with an obvious but very useful lemma that is used in all of the following propositions.
\begin{lemma} \label{lemma:lin_over l2_ball}
Given $0\neq a\in\reals^n$,
$$ \max\{\langle a, x \rangle : \|x\|_2=1,x\in\reals^n\}=
\max\{\langle a, x \rangle : \|x\|_2\leq1,x\in\reals^n\}=\|a\|_2,$$ with maximizer
$x^*=a/\|a\|_2.$
\end{lemma}
\begin{proof} Immediate from Cauchy-Schwarz inequality. \end{proof}

The following propositions make use of the following operator:
\begin{definition}
Given any $a\in\reals^n$, define  the operator $T_k:\reals^n\rightarrow\reals^n$ by
\[
T_k(a):=\argmin_x{\{\|x-a\|_2^2 : \|x\|_0\leq k, x\in\reals^n\}}.
\]
\end{definition}
This operator is thus the best $k$-sparse approximation of a given vector $a$. Despite the
nonconvexity of the constraint, it is easy to see that $(T_k(a))_i=a_i$ for the $k$ largest
entries (in absolute value) of $a$ and $(T_k(x))_i=0$ otherwise. In case the $k$ largest entries
are not uniquely defined, we select the smallest possible indices.

In other words, without loss of generality, with $a \in \reals^n$ such $|a_1|\geq\ldots\geq |a_n|$, we have
\[
(T_k(a))_i=\left\{
        \begin{array}{ll}
          a_i, & i\leq k; \\
          0, & \mbox{otherwise}.
        \end{array}
      \right.
\]
Computing $T_k(\cdot)$ only requires determining the $k^{th}$ largest number of a vector of $n$ numbers which can be done in $O(n)$ time \cite{Blum1973} and zeroing out the proper components in one more pass of the $n$ numbers.

The next proposition is an extension of Lemma \ref{lemma:lin_over l2_ball} to
$l_0$-constrained problems. This is the simple result that maximizing a linear function over the nonconvex set $\{x\in\reals^n:\|x\|_2=1, \|x\|_0\leq
k\}$ is equally simple and can be solved in $O(n)$ time.

\begin{proposition} \label{prop:lin_over_l2ball_card}
Given $ 0\neq a\in\reals^n$,
\begin{equation} \label{eq:lin_over_l2ball_card}
 \max_x{\{\langle a,x\rangle : \|x\|_2=1, \|x\|_0\leq k, x\in\reals^n\}} = \|T_k(a)\|_2
\end{equation}
with solution obtained at
\[x^*=\frac{T_k(a)}{\|T_k(a)\|_2}.\]
\end{proposition}
\begin{proof}
By Lemma \ref{lemma:lin_over l2_ball}, the optimal value of problem
(\ref{eq:lin_over_l2ball_card}) is $\sqrt{\sum_{i\in\mathcal{I}}{a_i^2}}$ for some subset of indices  $\mathcal{I}\subseteq[n]$ with $|\mathcal{I}|\leq k$. The set $\mathcal{I}$ that maximizes this value clearly contains the indices of the $k$ largest elements of the vector $|a|$. Thus, by definition of $T_k(a)$, solving problem (\ref{eq:lin_over_l2ball_card}) is
equivalent to solving
\[
\max_x{\{\langle x,T_k(a)\rangle : \|x\|_2=1, x\in\reals^n\}}
\]
from which the result follows by Lemma \ref{lemma:lin_over l2_ball}.
\end{proof}

Another version of this result gives the solution with a squared objective.
\begin{proposition} \label{prop:squared_lin_over_l2ball_card}
Given $a\in\reals^n$,
\[ \max_x{\{\langle a,x\rangle^2 : \|x\|_2=1, \|x\|_0\leq k, x\in\reals^n\}} = \|T_k(a)\|^2_2\]
with solution obtained at
\[x^*=\frac{T_k(a)}{\|T_k(a)\|_2}.\]
\end{proposition}
\begin{proof}
Notice that the optimal objective value in Proposition \ref{prop:lin_over_l2ball_card} is nonnegative, and hence the squared objective value here does not change the optimal solution to the problem with linear objective.  The result follows.
\end{proof}

%We next extend Proposition \ref{prop:lin_over_l2ball_card} to the $l_0$-penalized version.
%\begin{proposition} \label{prop:lin_l0_pen_over_l2ball}
%Given $a\in\reals^n,s>0$,
%\begin{equation}
%\max_x{\{a^Tx -s\|x\|_0: \|x\|_2=1\}} = \|T_{p^*}(a)\|_2-sp^*.
%\end{equation}
%with solution obtained at
%\[x^*=\frac{T_{p^*}(a)}{\|T_{p^*}(a)\|_2}\]
%where
%\[ p^*=\argmax_{p\in[n]}{\{\|T_p(a)\|_2-sp\}}. \]
%\end{proposition}
%\begin{proof}
%The $l_0$ penalized maximization problem can be reformulated as
%\[ \max_{p_\in[n]}{\{-sp+\max_x{\{a^Tx : \|x\|_2=1,\|x\|_0\leq p\}}\}} \]
%for which the inner maximization problem can be solved using Proposition \ref{prop:lin_over_l2ball_card} to get
%\[ \max_{p_\in[n]}{\{-sp+\|T_p(a)\|_2\}} \]
%which is a finite one dimensional problem.  The result trivially follows.
%\end{proof}\\
%Note that solving the problem in Proposition \ref{prop:lin_l0_pen_over_l2ball}
%requires computing the $T_p(\cdot)$ operator $n$ times.
%However, rather we can first sort the vector $a$ in $O(n\log{n})$ time
%and compute $p^*$ in one more pass of the sorted vector, thus computing the solution in $O(n\log{n})$ time.

The above propositions will serve to derive the novel schemes given in Section \ref{s:sparse_pca}.  The next result is a modification of Proposition \ref{prop:squared_lin_over_l2ball_card}.  It shows that the $l_0$-penalized version of maximizing a squared linear function yields a closed-form solution.

\begin{proposition} \label{prop:squared_l0_pen_over_l2ball}
Given $a\in\reals^n,s>0$,
\[ \max_x{\{\langle a,x\rangle^2 -s\|x\|_0: \|x\|_2=1, x\in\reals^n\}}=\displaystyle\sum_{i=1}^n{(a_i^2-s)_+}\]
is solved by
\[x_i^*=
\frac{a_i[\mbox{sgn}(a_i^2-s)]_+}{\sqrt{\sum_{j=1}^n{a^2_j[\mbox{sgn}(a_j^2-s)]_+}}}.
\]
\end{proposition}
\begin{proof}
Assume without loss of generality that $|a_1|\ge\ldots\ge|a_n|$. The problem can be rewritten as
\[ \max_{p\in[n]}{\{-sp+\max_x{\{\langle a,x\rangle^2: \|x\|_2=1,\|x\|_0\leq p, x\in\reals^n\}}\}}. \]
Using Proposition \ref{prop:squared_lin_over_l2ball_card}, the inner maximization in $x$ is solved at
\[
x_i^*=\left\{
        \begin{array}{ll}
          a_i/\sqrt{\sum_{j=1}^p{a^2_j}}, & i\leq p; \\
          0, & \mbox{otherwise},
        \end{array}
      \right.
\]
and the problem is equal to
\[ \max_{p\in[n]}{\{-sp+\|T_p(a)\|_2^2\}}=\max_{p\in[n]}{\{\displaystyle\sum_{i=1}^p{a_i^2}-sp\}} =\max_{p\in[n]}{\{\displaystyle\sum_{i=1}^p{(a_i^2-s)}\}}= \displaystyle\sum_{i=1}^n{(a_i^2-s)_+}. \]
Notice that the optimal $p$ is the largest index $i$ such that $a_i^2\ge s$
which makes the above expression for $x^*$ equivalent to the expression in the proposition.
\end{proof}

Our next two results are concerned with $l_1$-penalized/constrained optimization problems.
First, it is useful to recall the following well-known operators which are particular
instances of the so-called Moreau's proximal map \cite{More65}; see, for instance, \cite{CW05} for these results and many more. Given $a\in\reals^n$ and $W$ an $n\times
n$ diagonal matrix $W=\mbox{diag}(w),\; w\in \reals^n$ with positive entries $w_i$, let
$$
\|Wx\|_1:=\sum_{i=1}^n w_i |x_i|;\; \mathbb{B}_\infty^w :=\{x \in \reals^n: \|W^{-1}x
\|_\infty \leq 1 \}.
$$
Then,
\begin{eqnarray}\label{tres-proj}
S_w(a)&:=& \argmin_{x}\{\frac{1}{2}\|x-a\|^2_2 + \|Wx\|_1\}=(|a|- w)_{+}\mbox{sgn}(a), \; \label{stresh}\\
\Pi_{\mathbb{B}_\infty^w }(a) &:=& \argmin_{x}\{\|x-a\|_2: x \in \mathbb{B}_\infty^w \}=
\mbox{sgn}(a)\min\{w,|a|\}=a -S_w(a), \label{proj}
\end{eqnarray}
where $S_w(a)$ and $\Pi_{\mathbb{B}_\infty^w }(a)$ are respectively known as the soft-thresholding operator and the projection operator.

\begin{proposition} \label{prop:lin_weighted_l1_pen_over_l2ball}
For $a \in\reals^n$, $w \in \reals^n_{++}$, and $W=\mbox{diag}(w)$
\[ \max{\{\langle a, x\rangle - \|Wx\|_1: \|x\|_2\leq 1, x\in\reals^n\}}=
\sqrt{\displaystyle\sum_{i=1}^n{(|a_i|-w_i)_+^2}}=\|S_w(a)\| \]
 which is solved by \[x^*= S_w(a)/\|S_w(a)\|_2.\]
%\max{\{\langle a, x\rangle - \displaystyle\sum_{i=1}^n{w_i|x_i|}: \|x\|_2\leq 1\}}=
%\[x_i^*=\frac{\mbox{sgn}{(a_i)}(| a_i|-w_i)_+}{\sum_{l=1}^n{\sqrt{(| a_l|-w_l)^2_+}}}, \; i \in [n].\]
\end{proposition}
\begin{proof} By the H\"{o}lder inequality, we have $\|Wx\|_1 =
\max_{\|v\|_\infty \leq 1} \langle v, Wx \rangle=\max \{\langle z, x \rangle: z \in
\mathbb{B}_\infty^w \}.$ Using the latter, we obtain
\begin{eqnarray*}
\max\{\langle a, x\rangle - \|Wx\|_1: \|x\|_2\leq 1\}&=&\max_{\|x\|_2\leq 1}\min_{z\in
\mathbb{B}_\infty^w } \langle a - z, x \rangle \\
&=& \min_{z\in \mathbb{B}_\infty^w }\max_{\|x\|_2\leq 1}\langle a - z, x \rangle \\
&=& \min\{ \| a -z \|_2: z\in \mathbb{B}_\infty^w\}=\|S_w(a)\|_2,
\end{eqnarray*}
where the second equality follows from standard min-max duality \cite{Rock1970}, the third from Lemma \ref{lemma:lin_over l2_ball}, and the last one from using the relations (\ref{stresh})-(\ref{proj}), where the optimal $z^*=a-S_w(a)$ and $x^*$ follows from Lemma \ref{lemma:lin_over l2_ball}.
\end{proof}

We next turn to the $l_2/l_1$-constrained problem,
\begin{equation}\label{l1pca}
\max \{ \langle a, x \rangle  : \|x\|_2 \leq 1, \|x\|_1\leq k, x\in\reals^n\},
\end{equation}
and state a result similar to Proposition \ref{prop:lin_over_l2ball_card} for maximizing
a linear function over the intersection of the $l_2$ unit ball with an $l_1$ constraint.  While maximizing over the intersection of the $l_2$ unit ball with an $l_0$ ball has an analytic solution, here we need an additional simple one dimensional search to express the solution of (\ref{l1pca}) via its dual.

\begin{proposition} \label{prop:lin_over_l2ball_l1ball}
Given $a\in\reals^n$, we have
\begin{equation}\label{eq:l1_l2_constrained_one_dim}
\max\{\langle a, x \rangle  : \|x\|_2\leq 1, \|x\|_1\leq k, x\in\reals^n\}= \min_{\lambda\ge0}\{\lambda k + \|S_{\lambda e}(a)\|_2\},
\end{equation}
the right hand side being a  dual of (\ref{l1pca}). Moreover, if $\lambda^* $ solves the one-dimensional dual, then an optimal solution of (\ref{l1pca}) is given by $x^*(\lambda^*)$ where:
\begin{equation}\label{xl1pca}
x^*(\lambda) = S_{\lambda e}(a)/\|S_{\lambda e}(a)\|_2, \; (e\equiv (1, \ldots ,1)\in
\reals^n).
\end{equation}
\end{proposition}
\begin{proof}
Dualizing only the $l_1$ constraint, standard Lagrangian duality \cite{Rock1970} implies:
$$\max \{\langle a, x \rangle  : \|x\|_2 \leq 1, \|x\|_1\leq k, x\in \reals^n\}=\min \{\lambda k + \psi(\lambda):
\lambda \geq 0 \},$$ with
$$
\psi(\lambda):= \max_{\|x\|_2 \leq 1}\{\langle a, x \rangle -\lambda \|x\|_1\}= \|S_{\lambda
e}(a)\|_2
$$
where the last equality follows from Proposition \ref{prop:lin_weighted_l1_pen_over_l2ball} with $x^*(\lambda)$ as given in (\ref{xl1pca}).\end{proof}

The above propositions will be used to create simple algorithms for $l_0$ and $l_1$-constrained and penalized PCA.

\section{Sparse PCA via Conditional Gradient Algorithms} \label{s:sparse_pca}
This section details algorithms for solving the original $l_0$-constrained PCA problem
(\ref{eq:sparse_pca}) and its three modifications (\ref{eq:sparse_pca_l1constrained}), (\ref{eq:sparse_pca_l0penalized}), and (\ref{eq:sparse_pca_l1penalized}). Everything is developed using the simple tool box from Section \ref{s:props}. We derive novel algorithms as well as other known schemes that are shown to be particular realizations of the conditional gradient algorithm. A common and interesting appeal of all these algorithms is that they take
the shape of a generalized power method, i.e., they can be written as
\[
x^{j+1}=\frac{\mathcal{S}(Ax^j)}{\|\mathcal{S}(Ax^j)\|_2}
\]
where $\mathcal{S}:\reals^n\rightarrow\reals^n$ is a simple operator given in closed form or
that can be computed very efficiently.

Table \ref{table:summary} summarizes the different algorithms that are discussed throughout
this section.  Except for the alternating minimization algorithm for $l_1$-penalized PCA of
\cite{Zou06}, they are all particular realizations of the ConGradU algorithm with
an appropriate choice of the objective/constraints $(F,C)$.

\begin{table}[h!]
\begin{center}
\small{
\begin{tabular}{|l|l|l|l|l|}
\hline
  \textbf{Type} & \textbf{Iteration} & \textbf{Per-Iteration}&\textbf{References}\\
& &\textbf{Complexity} & \\ \hline &&& \\
$l_0$-constrained& $x^j=\frac{T_k((A+\frac{\sigma}{2}I_n)x^j)}{\|T_k((A+\frac{\sigma}{2}I_n)x^j)\|_2}$& $O(kn),O(mn)$ & novel\\ &&&\\\hline &&&\\
$l_1$-constrained&$x_i^{j+1}=\frac{\mbox{sgn}{(((A+\frac{\sigma}{2})x^j)_i)}(| ((A+\frac{\sigma}{2})x^j)_i|-\lambda^j)_+}{\sqrt{\sum_h{(| ((A+\frac{\sigma}{2})x^j)_h|-\lambda^j)_+^2}}}$ & $O(n^2),O(mn)$ & \cite{Witt2009} \\ &&&\\\hline &&&\\
$l_1$-constrained&$x_i^{j+1}=\frac{\mbox{sgn}{((Ax^j)_i)}(| (Ax^j)_i|-s^j)_+}{\sqrt{\sum_h{(| (Ax^j)_h|-s^j)_+^2}}}$\quad where &$O(n^2),O(mn)$ &\cite{Sigg2008} \\ &&& \\
&$s^j$ is $(k+1)$-largest entry of vector $|Ax^j|$& & \\ \hline &&&\\
$l_0$-penalized&$z^{j+1}=\frac{\sum_i{[\mbox{sgn}((b_i^Tz^j)^2-s)]_+(b_i^Tz^j)b_i}}{\|\sum_i{[\mbox{sgn}((b_i^Tz^j)^2-s)]_+(b_i^Tz^j)b_i}\|_2}$ & $O(mn)$&\cite{Shen2008,Jour2010} \\&&& \\\hline&&&\\
$l_0$-penalized&$x^{j+1}_i=\frac{\mbox{sgn}{(2(Ax^j)_i)}(|2(Ax^j)_i|-s\varphi_p'(|x^j_i|))_+}{\sqrt{\sum_h{(|2(Ax^j)_h|-s\varphi_p'(|x^j_h|))^2_+}}}$ &$O(n^2)$ &\cite{Srip2010} \\&&& \\\hline &&&\\
$l_1$-penalized&$y^{j+1}=\argmin_y{\{\sum_i{\|b_i-x^jy^Tb_i\|_2^2}+\lambda\|y\|_2^2+s\|y\|_1\}}$ &See Section \ref{ss:l1penalized_pca}&\cite{Zou06} \\&&&\\
&$x^{j+1}=\frac{(\sum_i{b_ib_i^T})y^{j+1}}{\|(\sum_i{b_ib_i^T})y^{j+1}\|_2}$& & \\&&&\\\hline&&&\\
$l_1$-penalized&$x_i^{j+1}=\frac{\mbox{sgn}{(((A+\frac{\sigma}{2})x^j)_i)}(| ((A+\frac{\sigma}{2})x^j)_i|-s)_+}{\sqrt{\sum_h{(| ((A+\frac{\sigma}{2})x^j)_h|-s)_+^2}}}$ &$O(n^2),O(mn)$ &novel \\ &&& \\\hline&&&\\
$l_1$-penalized&$z^{j+1}=\frac{\sum_i{(|b_i^Tz^j|-s)_+\mbox{sgn}(b_i^Tz^j)b_i}}{\|\sum_i{(|b_i^Tz^j|-s)_+\mbox{sgn}(b_i^Tz^j)b_i}\|_2}$ &$O(mn)$ &\cite{Shen2008,Jour2010} \\&&& \\\hline
\end{tabular} }

\end{center}
\caption{Sparse PCA Algorithms.  For each iteration,
$B\in\reals^{m\times n}$ is a data matrix with $A=B^TB$.
$b_i$ is the $i^{th}$ column of $B$
(except for the $l_1$-penalized PCA of \protect\cite{Zou06} where it is the $i^{th}$ row of $B$).  For iterations with two complexities, the first uses the covariance matrix $A$ and the second uses the decomposition $A=B^TB$ to compute matrix-vector products as $Ax$ or $B^T(Bx)$. Several iterations have two complexities, depending on whether data matrix $B$ is available.
The regularized $l_1$-constrained version of \protect\cite{Witt2009} is also novel.
The $l_0$ and $l_1$-penalized iterations of \protect\cite{Shen2008} require an $O(mn)$
transformation to recover a sparse $x^*$ from $z^*$. }
\label{table:summary}\end{table}

We stress that the first algorithm for $l_0$-constrained PCA
is the only algorithm that applies directly to the original and unmodified $l_0$-constrained PCA problem. The other algorithms are applied to modified problems
where tuning a parameter is needed to get an approximation to the desired $k$-sparse problem.  Unless otherwise specified, $A$ is only assumed to be symmetric and $A_\sigma$ is used to denote the regularized positive definite matrix.  The exceptions are only when we assume we are given a data matrix $B$ so that $A=B^TB$.

\subsection{$l_0$-Constrained PCA} \label{ss:l0constrained_pca}
In this section, we focus on the original $l_0$-constrained PCA problem
\begin{equation} \label{eq:sparse_pca2}
\max{\{x^TAx : \|x\|_2=1, \|x\|_0\leq k, x\in\reals^n\}}.
\end{equation}
We apply the ConGradU algorithm with the constraint set $C=\{x\in\reals^n:\|x\|_2=1, \|x\|_0\leq k\}$ and the convex objective (cf. Section \ref{s:formulations}):
$$
q_\sigma(x)= x^T(A + \sigma I)x=x^TA_\sigma x, \; \sigma \geq 0.
$$
When  $A \in S^n$ is already given positive semidefinite, the objective is already convex and there is no need for regularization and thus we simply set $\sigma =0$. The resulting main iteration of ConGradU reduces to,
\begin{equation} \label{eq:congradu_l0cons}
x^{j+1}=\argmax{\{\langle x ,A_\sigma x^j \rangle : \|x\|_2=1,
\|x\|_0\leq k, x\in\reals^n\}}=\frac{T_k(A_\sigma x^j)}{\|T_k(A_\sigma x^j)\|_2}, \; j=0,1, \ldots
\end{equation}
with $x^0\in C$ and where the second equality is due to Proposition \ref{prop:lin_over_l2ball_card}.

This  novel iteration  is obtained by  maximizing a continuously differentiable convex function over a compact nonconvex set, and by Theorem \ref{thm:fw_stationarity_convex_obj}, every one of its limit points converges to a stationary point. The complexity of each iteration requires computing $A_\sigma x^j$ where
$A_\sigma\in\symm^{n\times n}$ and $x^j$ is $k$-sparse so the matrix-vector product requires $O(nk)$
complexity. Computing the $T_k(\cdot)$ operator is $O(n)$ so each iteration is $O(nk)$. For very large problems where only a data matrix $B\in\reals^{m\times n}$ can be stored, the matrix-vector product $A_\sigma x^j$ is computed as $B^T(Bx^j)$, which requires complexity $O(mn)$, so each iteration is actually $O(mn)$.  The new scheme based on iteration (\ref{eq:congradu_l0cons}) is an extremely simple way to approach the original $l_0$-constrained PCA problem (\ref{eq:sparse_pca2}) for any given matrix $A_\sigma\in S^n$
and is the cheapest known non-greedy approach to $l_0$-constrained PCA.

%It can be generalized as well to the regularized case $\sigma\ge0$:
%\begin{equation}\begin{array}{ll}
%x^{j+1}&=\argmax{\{2x^TAx^j -\frac{\sigma}{2}(x^Tx-2x^Tx^j+x^{jT}x^j) : \|x\|_2=1, \|x\|_0\leq k\}}\\
%&=\argmax{\{2x^TAx^j +\sigma x^Tx^j : \|x\|_2=1, \|x\|_0\leq k\}}\\
%&=\argmax{\{x^T(A+\frac{\sigma}{2})x^j : \|x\|_2=1, \|x\|_0\leq k\}}\\
%&=\frac{T_k((A+\frac{\sigma}{2}I_n)x^j)}{\|T_k((A+\frac{\sigma}{2}I_n)x^j)\|_2}
%\end{array}\label{eq:l0_constrained_pca_iter}\end{equation}
%where the last equality is again due to Proposition \ref{prop:lin_over_l2ball_card}.

Thus, a very simple gradient algorithm (and not greedy heuristics) can be directly
applied to the desired problem. To put our novel scheme in perspective, let us recall the
work \cite{Mogh06b} which offers the following greedy scheme. Given a set of $k$ indices
for variables with nonzero values, $n-k$ subsets of indices are created by appending the $k$
indices with one from the $n-k$ remaining indices.  Then $n-k$ possible matrices are computed
from the $n-k$ groups of indices and the matrix with the maximum  eigenvalue gives the index
that provides the $k+1$-sparse PCA solution.  A full path of solutions can be computed for all
values of $k$,  but at a costly expense $O(n^4)$ (or up to $k$-sparsity in $O(kn^3)$), and with no statements on stationarity that we have.  The work \cite{Mogh06b} also produces a  branch-and-bound method for computing the exact
solution,   however, this method is amenable to only very small problems. The authors of \cite{dasp2008} considered what they call an \emph{approximate} greedy algorithm that generates an entire path of solutions up to $k$-sparsity in $O(kn^2)$ ($k\in\{1,\ldots,n\}$).  Rather than computing the $n-k$ maximum eigenvalues each iteration, their scheme computes $n-k$ dot products each iteration.  While the path is cheap for all solutions, it is expensive when only the $k$-sparse solution is desired for a single value $k$.  The total path (up to $k$-sparse solutions) using our computations can be computed in $O(kmn)$ (assuming finite convergence).  While the (approximate) greedy algorithms are more computationally expensive, as will be shown in Section \ref{s:experiments}, they do offer good practical performance (under measures to be discussed later).

Finally, we again stress the importance of considering such simple approaches where we can only provide convergence to stationary points. Recent convex relaxations for this problem \cite{dasp07,Luss2011}, while offering new insights to the
problem, have the same disadvantage that the gap to the optimal solution of problem
(\ref{eq:sparse_pca2}) cannot be computed (together, these methods can give primal-dual gaps).
Only a gap to the optimal solution of a relaxation (convex upper bound) is computed. Another
major disadvantage is that convex relaxations are not amenable to very large data sets as the
per-iteration complexity is $O(n^3)$ and they require far more iterations in practice.
Application of our scheme is limited only by storage of the data; only $nk$ entries of the covariance matrix are needed at each iteration.

%Notice that taking $\sigma>0$ is equivalent to adding a constant to the objective
%(since we are adding $\sigma\|x\|_2/2=\sigma/2$), and thus the same problem is actually solved.
%For $\sigma>0$, convergence is hence to a stationary point of the problem with a spectrum-shifted $A$,
%which is a stationary point of the same problem.  Hence, for $A\succeq0$,
%there is no benefit from regularization, however for $A\not\succeq0$, the following remark describes the benefit.
%\begin{remark} \label{rem:sparse_pca indefinite_A}
%If $A$ is not positive semidefinite, problem (\ref{eq:sparse_pca2}) can be reformulated equivalently as
%\begin{equation}
%\max{\{x^T(A+\lambda_{\max}(A))x : \|x\|_2=1, \|x\|_0\leq k\}}
%\end{equation}
%with convex objective since only a constant $\lambda_{\max}(A)x^x=\lambda_{\max}(A)$ has been added to the objective.
%Note this reformulation works for all problems with the constraint $\|x\|_2=1$.
%With a convex objective, the convergence results of Theorem \ref{thm:fw_stationarity_convex_obj} apply.
%\end{remark}

\subsection{$l_1$-Constrained PCA} \label{ss:l1constrained_pca}
In this section, we focus on the $l_1$-constrained PCA problem
\begin{equation}
\label{eq:sparse_pca_l1constrained2} \max{\{x^TAx : \|x\|_2 \leq 1,\|x\|_1\le \sqrt{k}, x\in\reals^n\}}.
\end{equation}
In \cite{Luss2011}, we provide a convex relaxation and corresponding algorithm with a per-iteration complexity of $O(n^3)$, but here,
we are again only interested in much less computationally expensive methods. We present two recent algorithms that have been proposed through different motivations and show that they can be recovered through our framework.\\

\noindent{\bf An Alternating Maximization Scheme}\\
Recently, Witten et al. \cite{Witt2009} have considered the sparse Singular Value Decomposition (SVD) problem
\[
\max_{x,y}{\{x^TBy : \|x\|_2=1,\|y\|_2=1,\|x\|_1\leq k_1,\|y\|_1\leq k_2, x,y\in\reals^n\}},
\]
where $B\in\reals^{m\times n}$ is the data matrix and $k_1,k_2$ are positive integers. They
recognized that the objective is linear in $x$ with fixed $y$ (and vice-versa) and that the
problems with $x$ or $y$ fixed can be easily solved (see Proposition \ref{prop:lin_over_l2ball_l1ball}). This fact motivates them to propose a simple cheap
alternating maximization scheme which in turn can be used to solve $l_1$-constrained PCA.
However, the work \cite{Witt2009} does not recognize this as yet another instance of the conditional gradient algorithm nor does it provide any convergence results.  We show below how to simply recover this algorithm under our framework by applying ConGradU directly to problem (\ref{eq:sparse_pca_l1constrained2}) for which the convergence claims of Theorem \ref{thm:fw_stationarity_convex_obj} hold.

Thanks to the results established in Section \ref{s:formulations}, we derive it here with the additional regularization term, i.e., with $q_\sigma(x)=x^TA_\sigma x, \; \sigma \geq 0$  for an arbitrary matrix $A\in S^n$ (with $\sigma=0$ when $A$ is already given positive semidefinite). Applying ConGradU, the scheme reads:
\begin{equation}
x^{j+1}=\argmax{\{\langle A_\sigma x, x^j\rangle : \|x\|_2\leq 1, \|x\|_1\leq\sqrt{k}, x\in\reals^n\}}.
\end{equation}
The above iteration maximizes a continuously differentiable convex function over a compact set so Theorem \ref{thm:fw_stationarity_convex_obj} can be applied to show that every limit point of the sequence converges to a stationary point. Furthermore, thanks to  Proposition \ref{prop:lin_over_l2ball_l1ball}, this scheme reduces to the iteration
\begin{equation} \label{eq:l1_constrained_pca_iter}
x^{j+1}=\frac{S_{\lambda^je}(A_\sigma x^j)}{\|S_{\lambda^je}(A_\sigma x^j)\|_2}
\end{equation}
where $\lambda^j$ is determined by solving (\ref{eq:l1_l2_constrained_one_dim}) with $a:=A_\sigma x^j$ (cf. Proposition \ref{prop:lin_over_l2ball_l1ball}).  The most expensive operation at each iteration is computing $A_\sigma x^j$ where $A_\sigma\in\symm_+^n$ and $x^j\in\reals^n$ so the per-iteration complexity is $O(n^2)$.\\
%Notice again here that regularization only helps when $A$ is not positive semidefinite
%(see Remark \ref{rem:sparse_pca indefinite_A}), because it is necessary to show convergence to a stationary point.

\noindent{\bf The Expectation-Maximization Algorithm}\\
Another recent approach to problem (\ref{eq:sparse_pca_l1constrained2}) is developed in
\cite{Sigg2008} which they motivate as an Expectation-Maximization (EM) algorithm for sparse PCA.
Motivation comes from their derivation of computing principal components using EM for a probabilistic
 version of PCA, which is actually equivalent to the power method.
 The authors solve $l_1$-constrained PCA, however also want
 to enforce that $\|x\|_0=k$ at each iteration as well. Their algorithm can be written as
\[
x^{j+1}=\frac{S_{\lambda^je}(A_\sigma x^j)}{\|S_{\lambda^je}(A_\sigma x^j)\|_2}
\]
where $\lambda^j$ is the $(k+1)$-largest entry of vector $|A_\sigma x^j|$.
 Note that the iteration form is identical to that above for the alternating maximization scheme,
 except for the computation of $\lambda^j$.
Thus, each iteration can be interpreted as solving
\begin{equation}
x^{j+1}=\argmax{\{\langle A_\sigma x, x^j\rangle : \|x\|_2=1, \|x\|_1\leq \sqrt{k^j}, x\in\reals^n\}},
\end{equation}
where $k^j$ is chosen at each iteration specifically so that $x^{j+1}$ is $k$-sparse, and can easily be seen to be a variant of ConGradU.  Enforcing $\lambda^j$ to be the $(k+1)$-largest entry of vector $|A_\sigma x^j|$ implicitly sets $k^j$ to a value that achieves $k$-sparsity in $x^{j+1}$. While this choice of thresholding enforces exactly $k$ nonzero entries, the iteration becomes heuristic and neither applies to the true $l_0$ or $l_1$-constrained problem.  It is cheap, with the major computation being to compute $A_\sigma x^j$, and performs well in practice as shown in Section \ref{s:experiments}. However, unlike our other iterations, there are no convergence results for this heuristic.

\subsection{$l_0$-Penalized PCA} \label{ss:l0penalized_pca}
We next consider the $l_0$-penalized PCA problem
\begin{equation} \label{eq:sparse_pca_l0penalized2}
\max{\{x^TAx - s\|x\|_0 : \|x\|_2\leq1, x\in\reals^n\}}
\end{equation}
which has received most of the recent attention in the literature. We  describe two recent algorithms to this problem and show again that they are direct applications of ConGradU.\\

\noindent{\bf Exploiting Positive Semidefiniteness of $A$}\\
The first approach due to \cite{Jour2010} assumes that $A$ is positive semidefinite (i.e., $A=B^TB$ with $B\in\reals^{m\times n}$) and writes problem
(\ref{eq:sparse_pca_l0penalized2}) as
\begin{equation} \label{eq:sparse_pca_l0penalized_nesterov}
\max{\{\|Bx\|_2^2 - s\|x\|_0 : \|x\|_2 \leq 1, x\in\reals^n\}}.
\end{equation}

The objective is neither concave nor convex. First, using the simple fact (consequence of Lemma \ref{lemma:lin_over l2_ball}) $$\|Bx\|_2^2=\max_{\|z\|_2
\leq 1}\{ \langle z, Bx \rangle^2 \},$$ the problem is equivalent to
$$
\max_{\|x\|_2 \leq 1}\max_{\|z\|_2\leq 1}\{ \langle z, Bx \rangle^2\  - s\|x\|_0
\}=\max_{\|z\|_2 \leq 1}\max_{\|x\|_2\leq 1}\{ \langle B^Tz, x \rangle^2\  - s\|x\|_0\}.
$$

Thus, we can now apply Proposition \ref{prop:squared_l0_pen_over_l2ball} to the inner
minimization in $x$ and then get
\begin{equation} \label{eq:sparse_pca_l0penalized_nesterov2}
\max_{x\in\reals^n}{\{\|Bx\|_2^2 - s\|x\|_0 : \|x\|_2\leq1\}}=\max_{z\in\reals^m}{\{\displaystyle\sum_{i=1}^n[\langle b_i,z\rangle^2-s]_+ :
\|z\|_2\leq1\}}
\end{equation}
where $b_i\in\reals^m$ is the $i^{th}$ column of $B$.  This reformulation was previously derived in \cite{dasp2008}, where the authors also provided a convex relaxation. Note that the reformulation operates in
the space $\reals^m$ rather than $\reals^n$.  Since the objective function
$f(z):=\sum_i[\langle b_i,z\rangle^2-s]_+$ is now clearly convex, we can apply ConGradU.
Noting that a subgradient of $f(z)$ is given by
\[
2\displaystyle\sum_{i=1}^n{[\mbox{sgn}(\langle b_i,z\rangle^2-s)]_+(\langle b_i,z\rangle)b_i},
\]
the resulting iteration (using Lemma \ref{lemma:lin_over l2_ball}) yields:
\begin{equation} \label{eq:l0_journee_iter}
z^{j+1}=\frac{\sum_i{[\mbox{sgn}((\langle b_i,z^j\rangle)^2-s)]_+(\langle b_i,z^j\rangle)b_i}}
{\|\sum_i{[\mbox{sgn}((\langle b_i,z^j\rangle)^2-s)]_+(\langle b_i,z^j\rangle)b_i}\|_2},
\end{equation}
and the convergence results for the nonsmooth case of Theorem
\ref{thm:fw_stationarity_convex_obj} apply.

This is exactly the algorithm recently derived in \cite{Jour2010}.  Note that an $O(mn)$
transformation is then needed via Proposition \ref{prop:squared_l0_pen_over_l2ball} to recover
the solution $x$ of the original problem (\ref{eq:sparse_pca_l0penalized_nesterov}).  This is
the first cheap ($O(mn)$ per-iteration complexity) and nongreedy approach for directly solving
the $l_0$-penalized problem. As with \cite{Witt2009} for $l_1$-constrained PCA,
\cite{Shen2008} approach $l_0$-penalized PCA via $l_0$-penalized SVD. After
modifications to write it out for $l_0$-penalized PCA, the resulting iteration of their
paper is equivalent to iteration (\ref{eq:l0_journee_iter}). However, they did not offer a
derivation or state the convergence properties given in \cite{Jour2010}.\\

\noindent{\bf Approximating the $l_0$-penalized Problem}\\
As explained in Section \ref{s:formulations}, we consider the $l_0$-penalized problem whereby we use an approximation to the $l_0$ norm, that is, we consider the problem of maximizing a convex function:
\begin{equation} \label{eq:sparse_pca_smoothed_gen}
\max_x \{ x^TA_\sigma x + g(|x|): \|x\|_2\leq 1, x\in \reals^n \}
\end{equation}
where the convex function $g$ is defined by
$$
g(z):= - s \sum_{i=1}^n \varphi_p(z_i), \; s >0,
$$
with $\varphi_p$ concave satisfying the premises given in Section \ref{s:formulations} and $A_\sigma=A+\sigma I, \sigma\ge0$.  Applying ConGradU to this problem, as shown in Section \ref{s:cond_grad_method}, using Proposition \ref{prop:nonsmooth_condtional_gradient} reduces the iteration to
the following weighted $l_1$-norm maximization:
\begin{equation}
x^{j+1}=\argmax{\{\langle A_\sigma x,x^j\rangle-\sum_i{w^j_i|x_i|} : \|x\|_2=1\}}
\end{equation}
where $w^j_i=s\varphi_p'(|x^j_i|)$.

Proposition \ref{prop:lin_weighted_l1_pen_over_l2ball} shows that this problem can be solved
in closed form so that the conditional gradient algorithm becomes
\[
x^{j+1}=\frac{S_{w^j}(A_\sigma x^j)}{\|S_{w^j}(A_\sigma x^j)\|_2}
\]
where again $w^j_i=s\varphi_p'(|x^j_i|)$. Theorem
\ref{thm:fw_stationarity_convex_obj} regarding convergence of every limit point
of the resulting sequence to a stationary point again applies.
The per-iteration complexity of this iteration is $O(n^2)$, which is reduced to $O(mn)$
if we have the factorization $A_\sigma=B^TB$.

Depending on the choice of  $\varphi$ (cf. Section \ref{s:formulations}), we thus have a
family of algorithms for solving problem (\ref{eq:sparse_pca_l0penalized2}). For example, with
the Example \ref{exphi} (b) given in Section \ref{s:formulations},
%\[
%\varphi_p(|t|)=\frac{\log{(1+|t|/p)}}{\log{(1+1/p)}}
%\]
we obtain the recent algorithm of \cite{Srip2010} which was derived there by
applying what is called the minorization-maximization method, a seemingly different approach; they consider $A\in\symm^n$ and represent the objective as a difference of convex functions plus the penalization: $x^TAx-s\sum_i{\varphi_p(|x_i|)}=x^TA_\sigma x-\sigma x^Tx-s\sum_i{\varphi_p(|x_i|)}$.  The objective is minorized by linearizing the convex term $x^TA_\sigma x-s\sum_i{\varphi_p(|x_i|)}$ resulting in a concave lower bound that is maximized.  When $\sigma=0$, this is identical to using the conditional gradient algorithm ConGradU, which is only one example of a minorization-maximization method.  For more on the minorization-maximization technique and its connection to gradient methods see the recent work
\cite{Beck2010} and references therein.  We also note that \cite{Srip2010}
derived their algorithm for the sparse generalized eigenvalue (GEV) problem
\begin{equation} \label{eq:sparse_gev}
\max{\{x^TAx : x^TBx \leq 1, \|x\|_0\leq k, x\in\reals^n\}}
\end{equation}
where $A\in\symm^n$ and $B\in\symm^n_{++}$, which includes as a special case the sparse PCA problem when $B$ is the identity matrix. The resulting algorithm of \cite{Srip2010} for this problem
requires computing a matrix pseudoinverse, and is much more computationally expensive (and not
amenable to extremely large data sets) than the same algorithm for sparse PCA. Moreover, using the results of Section \ref{s:formulations}, clearly the general iteration for indefinite $A$ need not be considered and sparse GEV can always be approached with a closed-form conditional gradient algorithm which still requires computing a matrix pseudoinverse (the closed-form iteration is derived in \cite{Srip2010}).

\subsection{$l_1$-Penalized PCA} \label{ss:l1penalized_pca}
Consider the $l_1$-penalized PCA problem
\begin{equation}
\label{eq:sparse_pca_l1penalized2} \max{\{x^TAx -s\|x\|_1 : \|x\|_2=1, x\in\reals^n\}}.
\end{equation}
This problem has a nonconvex objective. The work \cite{dasp07} provides a convex relaxation that
is solved via semidefinite programming with a per-iteration complexity of $O(n^3)$, but here,
we are again only interested in much less computationally expensive methods. We describe two
methods that exploit positive semidefiniteness of $A$, along with a novel scheme that does not require $A\in\symm^+_n$.\\

\noindent{\bf Reformulation with a Convex Objective}\\
To apply ConGradU, we need either a convex objective or, as shown, an objective of the
form $f(x)+g(|x|)$ with $f,g$ satisfying certain properties (cf. Section
\ref{s:formulations}). Exploiting the fact that $A\in S^n_{+}$, with $A=B^TB, B \in
\reals^{m\times n}$, an equivalent reformulation of the original $l_0$-constrained PCA problem can use the square root objective $\|Bx \|_2$, and hence the corresponding $l_1$-penalized PCA problem reads, instead of (\ref{eq:sparse_pca_l1penalized2}), as
$$ \max{\{\|Bx\|_2 -s\|x\|_1 :
\|x\|_2\leq1, x\in\reals^n\}}.$$
One can think of this as replacing the objective in our original $l_0$-constrained PCA problem (\ref{eq:sparse_pca}) with $\sqrt{x^TAx}=\|Bx\|_2$ (which is an equivalent problem) and then making the modifications for $l_1$-penalized PCA. The objective remains problematic and is neither convex nor concave. However, using again the
fact that $\|Bx\|_2=\max\{ \langle z, Bx \rangle: \|z\|_2\leq 1,z\in\reals^m\}$, the problem
reads
$$\max\{\|Bx\|_2 -s\|x\|_1 : \|x\|_2\leq1, x\in\reals^n\}= \max_{\|z\|_2\leq 1}\max_{\|x\|_2 \leq 1}\{\langle z, Bx \rangle
-s \|x\|_1\}.$$

Thus, applying Proposition \ref{prop:lin_weighted_l1_pen_over_l2ball}, the inner maximization
with respect to $x$ can be solved explicitly and the problem can be reformulated as maximizing
a convex objective, and we obtain:
$$
\max_{x\in\reals^n}{\{\|Bx\|_2 -s\|x\|_1 : \|x\|_2\leq1\}}= \max_{z\in\reals^m}{\{\displaystyle\sum_{i=1}^n(|b_i^Tz|-s)^2_+ :
\|z\|_2\leq1\}}, $$
where $b_i$ is the $i^{th}$ column of $B$.

We can now apply ConGradU to the convex (for similar reasons as for the $l_0$-penalized case in Section \ref{ss:l0penalized_pca}) objective $f(z)=\sum_i[|b_i^Tz|-s]^2_+$, and
for which our convergence results for the nonsmooth case hold true.  A subgradient of $f$ is
given by
\[
2\displaystyle\sum_{i=1}^n{(|b_i^Tz|-s)_+\mbox{sgn}(b_i^Tz)b_i},
\]
and, using Lemma \ref{lemma:lin_over l2_ball}, the resulting iteration is
\begin{equation} \label{eq:l1_journee_iter}
z^{j+1}=\frac{\sum_i{(|b_i^Tz^j|-s)_+\mbox{sgn}(b_i^Tz^j)b_i}}{\|\sum_i{(|b_i^Tz^j|-s)_+\mbox{sgn}(b_i^Tz^j)b_i}\|_2}.
\end{equation}

This is exactly the other algorithm recently derived in  \cite{Jour2010}. Note that to recover the solution to problem (\ref{prop:lin_weighted_l1_pen_over_l2ball}) needs an $O(mn)$
transformation via Proposition \ref{prop:lin_weighted_l1_pen_over_l2ball}.  This algorithm has
an $O(mn)$ per-iteration complexity. Note also that this algorithm  was stated earlier in
\cite{Shen2008} but no such derivation or convergence results were given.\\

\noindent{\bf A Novel Direct Approach}\\
We next derive a novel algorithm for problem (\ref{eq:sparse_pca_l1penalized2}) by directly
applying the conditional gradient algorithm.  Indeed, problem (\ref{eq:sparse_pca_l1penalized2}) reads as maximizing $f(x)+g(|x|)$ with $f(x)$ convex, $g(x)$ convex, differentiable,
and monotone decreasing, with $f(x)=x^TA_\sigma x$ and $g(u)=-\sum_i{u_i}$ where $A_\sigma$ is as
previously defined.  Applying the ConGradU algorithm and Proposition
\ref{prop:nonsmooth_condtional_gradient} leads to the iteration
\begin{equation}
x^{j+1}=\argmax{\{\langle A_\sigma x^j,x\rangle-s\|x\|_1 : \|x\|_2=1\}},
\end{equation}
which by Proposition \ref{prop:lin_weighted_l1_pen_over_l2ball} reduces to
\begin{equation} \label{eq:l1_penalized_pca_iter}
x^{j+1}=\frac{S_{se}(A_\sigma x^j)}{\|S_{se}(A_\sigma x^j)\|_2},
\end{equation}
where $e$ is a vector of ones.  Theorem \ref{thm:fw_stationarity_convex_obj} applies, showing that any limit point of this iteration is a stationary point of the $l_1$-penalized PCA problem.

The matrix-vector product $Ax^j$ is the main computational cost so the per-iteration
complexity is $O(n^2)$ (or $O(mn)$ if computing $B^T(Bx^j)$). This approach can handle matrices $A$ that are not positive
semidefinite (by taking $\sigma>0$) and has stronger convergence results than the conditional gradient method applied to the reformulation of \cite{Jour2010}, i.e., this approach is equivalent to applying ConGradU with a differentiable objective function (by Proposition \ref{prop:lin_weighted_l1_pen_over_l2ball}) and thus satisfies part (c) of Theorem \ref{thm:fw_stationarity_convex_obj}.  \cite{Jour2010} apply ConGradU to a different nondifferentiable formulation for which our theory does not apply.

For the sake of completeness, we end this section by mentioning one of the earlier cheap schemes for sparse PCA, even though it does not fall into the category of directly applying ConGradU.\\

\noindent{\bf An Alternating Minimization Scheme}\\
One of the earlier cheap approaches to sparse PCA, specifically
for $l_1$-penalized PCA, is proposed in \cite{Zou06} (SPCA).  While they generalize all results to
multiple factors, we only discuss the one factor case. They pose sparse PCA as an $l_1/l_2$-regularized
regression problem, specifically
\begin{equation} \label{eq:spca_lasso}
(x^*,y^*)=\argmin_{x,y}{\{\displaystyle\sum_{i=1}^m{\|b_i-xy^Tb_i\|_2^2}+\lambda\|y\|_2^2+s\|y\|_1 : \|x\|_2^2=1, x,y\in\reals^n\}}
\end{equation}
where $\lambda$ and $s$ are the $l_2$ and $l_1$ regularization parameters, respectively, and $B\in\reals^{m\times n}$ is a data matrix with rows $b_i\in\reals^n$.  When $s=0$, they show that $y^*$ is proportional to the leading eigenvector of $B^TB$.  Indeed, when $s=0$, problem (\ref{eq:spca_lasso}) can be recast as a classical maximum eigenvalue problem in $x$:
\begin{equation} \label{eq:spca_maxeig}
x^*=\argmax_x{\{x^TB^TB(B^TB+\lambda I_n)^{-1}B^TBx : \|x\|^2_2=1, x\in\reals^n\}}
\end{equation}
by first solving for $y$ (simple algebra shows $y^*=(B^TB+\lambda I_n)^{-1}B^TBx)$ and plugging $y^*$ into (\ref{eq:spca_lasso}).  It is easy to show that the $x^*$ that solves problem (\ref{eq:spca_maxeig}) is equal to the leading eigenvector of $B^TB$ for all $\lambda\ge0$, and thus, for the purposes of finding the leading eigenvector, we do not need to regularize the matrix (i.e., set $\lambda=0$).

Problem (\ref{eq:spca_lasso}) uses an $l_1$ penalty, known as a LASSO penalty, in order to induce sparsity on $y$ resulting in an approximate sparse leading eigenvector $y^*/\|y^*\|_2$.  An alternating minimization scheme in $x$ and $y$ is proposed to solve problem (\ref{eq:spca_lasso}). For fixed $y$, we have
\[
\sum_i{\|b_i-xy^Tb_i\|_2}=\sum_i{(b_i^Tb_i-2(y^Tb_i)b_i^Tx+(y^Tb_i)^2x^Tx)}=-2y^T(\sum_i{b_ib_i^T})x+C
\]
where $C$ is a constant (using the constraint $\|x\|_2=1$), so that the minimizer $x^*$ is
solved by maximizing a linear function over the unit sphere which, by Lemma
\ref{lemma:lin_over l2_ball}, is easily solved in closed-form.  For fixed $x$, the minimizer
$y^*$ is found by solving an unconstrained minimization problem of the form
$\|\cdot\|_2^2+s\|\cdot\|_1$ (also known as the \emph{elastic net} problem). This problem can
be solved efficiently for fixed $s$ using fast first-order methods such as FISTA
\cite{Beck2009} or for a full path of values for $s$ using LARS \cite{Efr2004}. Thus,
\cite{Zou06} solve a nonconvex problem in two variables using alternating minimization.
While this scheme is computationally inexpensive compared to convex relaxations, it is not as
cheap as the schemes we are considering due to the subproblem with fixed $x$, and no
convergence results have been derived for it.

\section{Experiments} \label{s:experiments}
Thus far, various algorithms have been provided with the goal of learning sparse rank one approximations.  In this section, these different methods are compared.  The algorithms considered here are $l_0$-constrained PCA (novel iteration), an approximate greedy algorithm \cite{dasp2008}, GPowerL1 ($l_1$-penalized PCA of \cite{Jour2010}), GPowerL0 ($l_0$-penalized PCA of \cite{Jour2010}), Expectation-Maximization ($l_1$-constrained PCA of \cite{Sigg2008}), and thresholding (select $k$ entries of principal eigenvector with largest magnitudes).  We also consider an exact greedy algorithm and the optimal solution (via exhaustive search) for small dimensions ($n=10$).

The goal of these experiments is two-fold.  Firstly, we demonstrate that the various algorithms give very similar performance.  The measure of comparison used is the proportion of variance explained by a sparse vector versus that explained by the true principal eigenvector, i.e., the ratio $x^TAx/v^TAv$ where $x$ is the sparse eigenvector and $v$ is the true principal eigenvector of $A$. The second goal is to solve very large sparse PCA problems.  The largest dimension we approach is $n=50000$, however, as discussed above, the ConGradU algorithm applied to $l_0$-constrained PCA has very cheap $O(mn)$ iterations and is limited only by storage of a data matrix.  Thus, on larger computers, extremely large-scale sparse PCA problems (much larger than those solved even here) are also feasible.

Note that we do not compare against all algorithms listed in Table \ref{table:summary}.  In particular, SPCA \cite{Zou06} was already demonstrated to be computationally more expensive, as well to provide inferior performance, to GPowerL1 and GPowerL0 \cite{Jour2010}.  The $l_1$-constrained PCA method of \cite{Witt2009} and $l_0$-penalized PCA method of \cite{Srip2010} are also cheap methods that give similar performance (learned from experiments not shown in this paper) to the algorithms in our experiments. Finally, note that, for all experiments, we do a postprocessing step in which we compute the largest eigenvector of the data matrix in the $k$-dimensional subspace that is discovered by the respective methods.

% RL: is this paragraph necessary
All experiments were performed in MATLAB on a PC with 2.40GHz processor with 3GB RAM. Codes from the competing methods were downloaded from URL's available in the corresponding references.
 Slight modifications were made to do singular value decompositions rather than eigenvalue decompositions
  in order to deal with much smaller $m\times n$ data matrices rather than $n \times n$ covariance matrices.  We first demonstrate performance on random matrices and follow with a text data example.

\subsection{Random Data}\label{ss:experiments_random}
We here consider random data matrices $F\in\reals^{m\times n}$ with $F_{ij}\sim N(0,1/m)$.
It was already shown in literature on greedy methods \cite{dasp2008} and convex relaxations \cite{dasp07,Luss2011}
that random matrices of the form $xx^T+U$ where $U$ is uniformly distributed noise are \emph{easy} examples.
 Results here show that taking sparse eigenvectors of the matrix $F^TF$ is also relatively easy.

The experiments consider $n=10$ $(m=6)$ and $n=5000, 10000, 50000$ (each with $m=150$),
each using 100 simulations.  We consider $l_0$-constrained PCA with $k=2,\ldots,9$
for $n=10$ and $k=5,10,\ldots,250$ for the remaining tests.
The optimal solution (found by exhaustive search) and the exact
greedy algorithm (too computationally expensive for high dimensions)
are only used when $n=10$.

\begin{figure}[h!] \begin{center}
  \begin{tabular} {cc}
     \psfrag{title}[t][t]{\small{$m=6,n=10$}}
     \psfrag{perc}[b]{\small{Proportion of explained variance}}
     \psfrag{k}[t]{\small{Sparsity}}
 \includegraphics[width=0.49\textwidth]{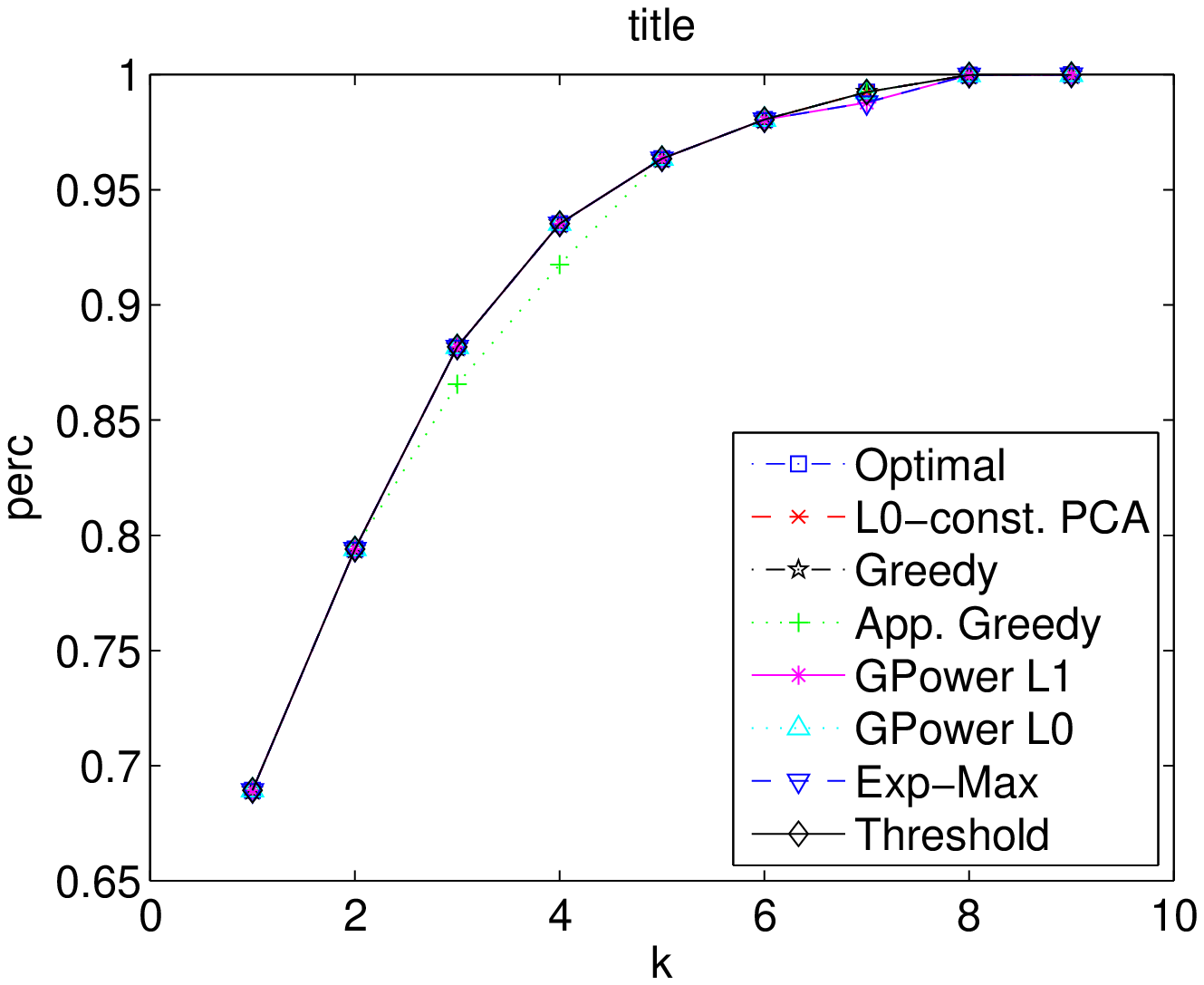}&
      \psfrag{title}[t][t]{\small{$m=150,n=5000$}}
     \psfrag{perc}[b]{\small{Proportion of explained variance}}
     \psfrag{k}[t]{\small{Sparsity}}    \includegraphics[width=0.49\textwidth]{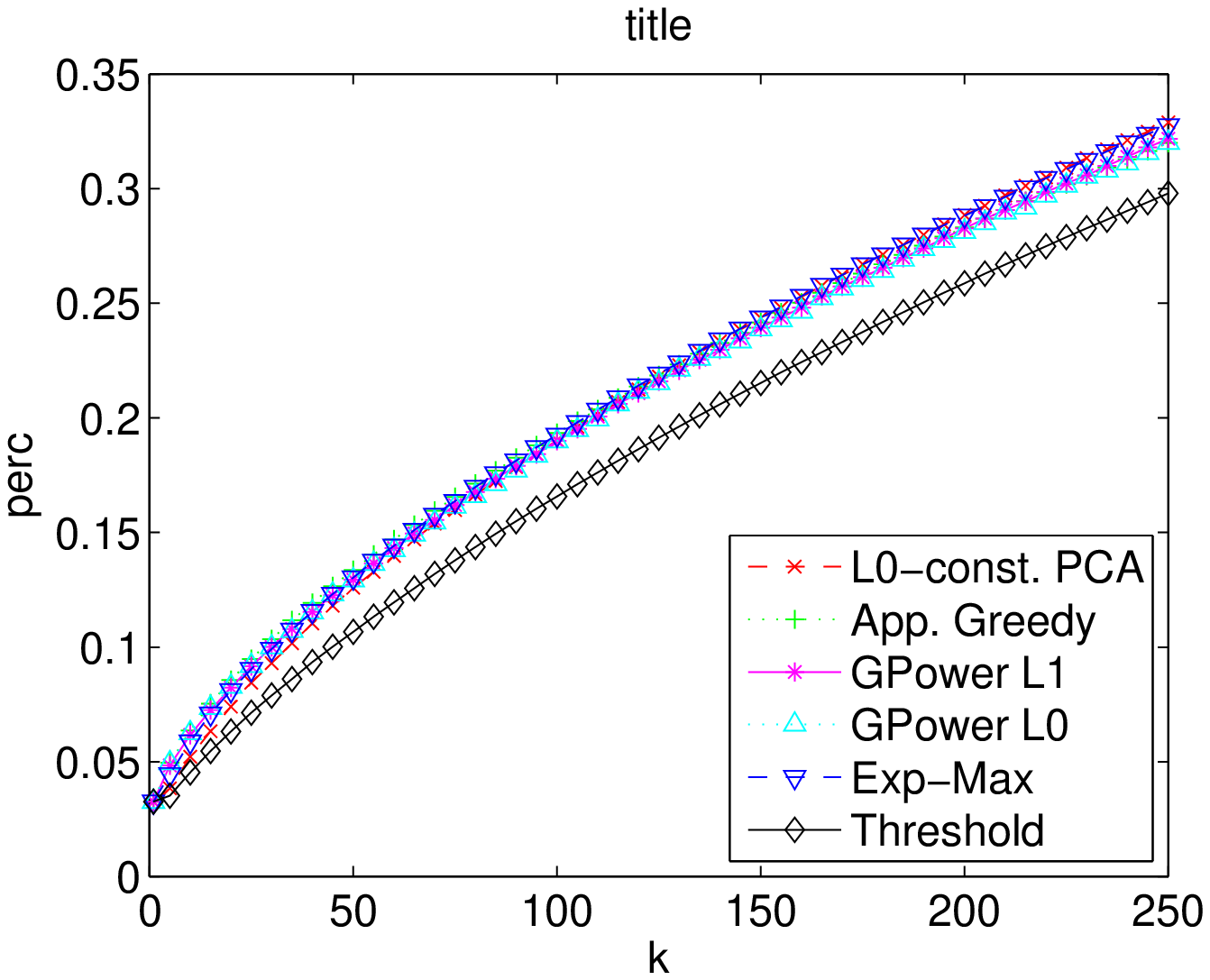}\\\\
     \psfrag{title}[t][t]{\small{$m=150,n=10000$}}
     \psfrag{perc}[b]{\small{Proportion of explained variance}}
     \psfrag{k}[t]{\small{Sparsity}}
\includegraphics[width=0.49\textwidth]{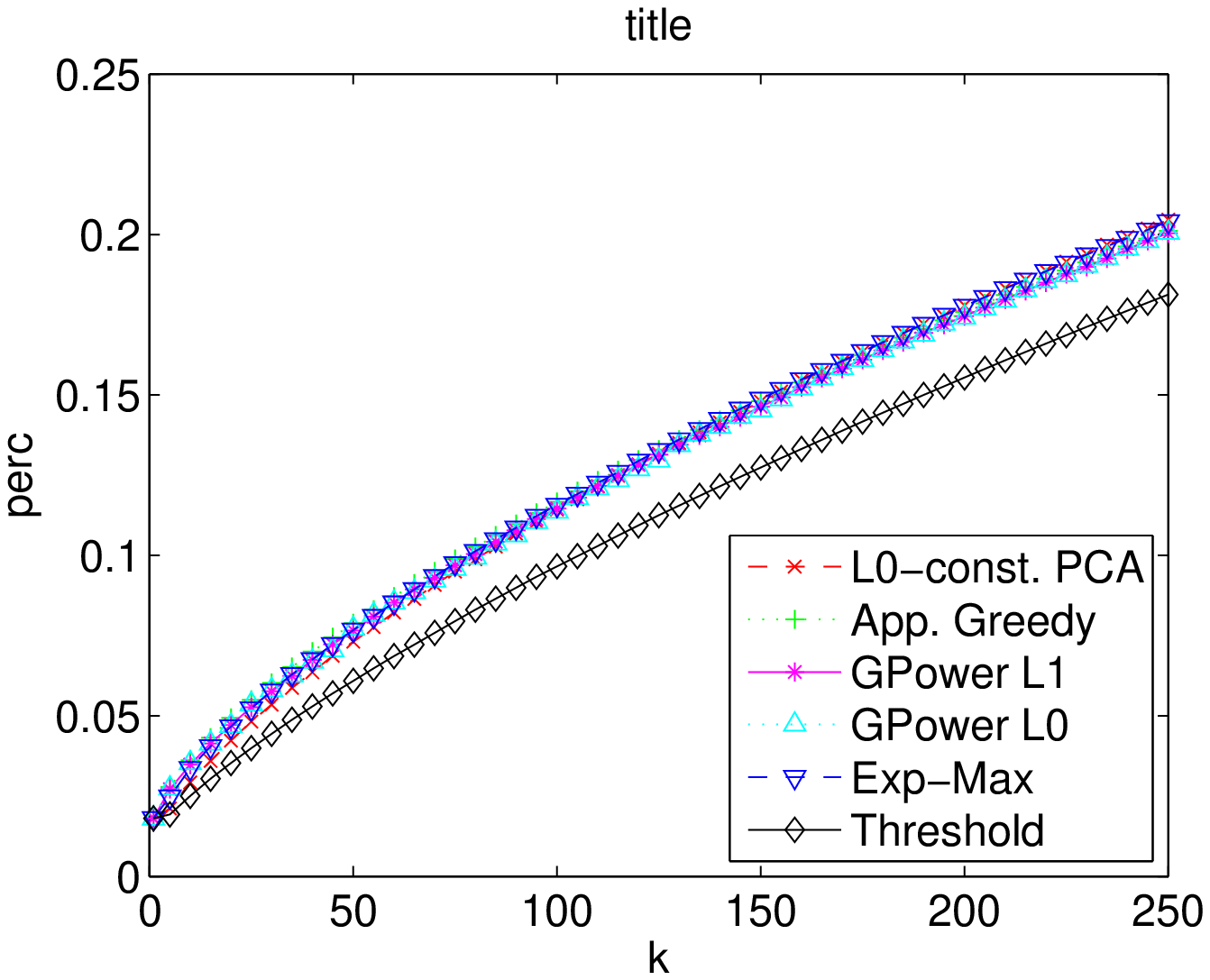}&
     \psfrag{title}[t][t]{\small{$m=150,n=50000$}}
     \psfrag{perc}[b]{\small{Proportion of explained variance}}
     \psfrag{k}[t]{\small{Sparsity}}
\includegraphics[width=0.49\textwidth]{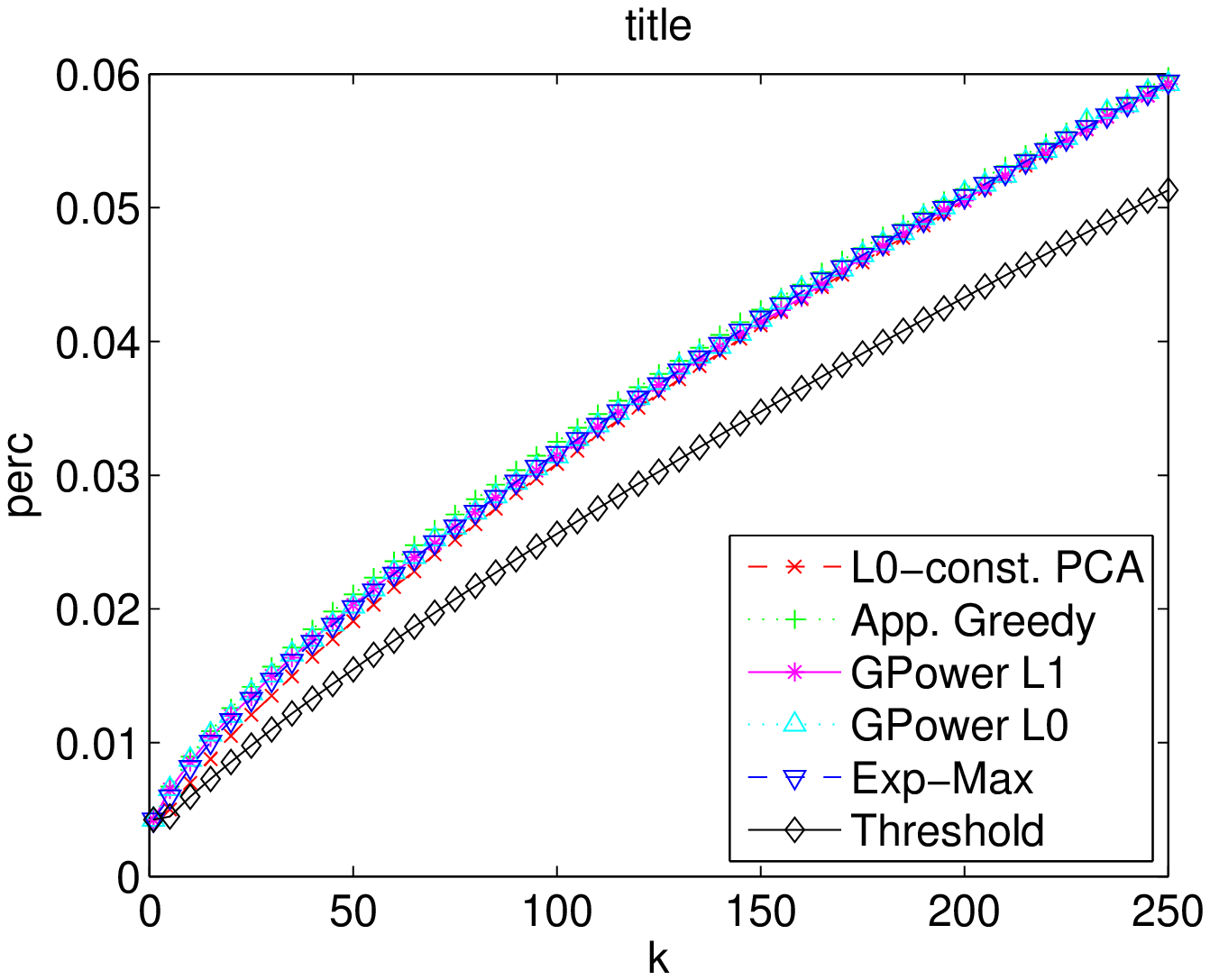}
\end{tabular}
\caption{Plots show the average percent of variance explained by the sparse eigenvectors found by several algorithms, i.e., the ratio $x^T(F^TF)x/v^T(F^TF)v$ where $x$ is the sparse eigenvector and $v$ is the true first eigenvector.  $F$ is an $m\times n$ matrix with $F_{ij}\sim N(0,1/m)$.  Sparsities of $2,\ldots,9$ are computed for $n=10$ and $5,10,\ldots,250$ for the remaining experiments.  100 simulations are used to produce all results.}
\label{fig:percGauss}\end{center}\end{figure}

Figure \ref{fig:percGauss} compares performance of the various algorithms.
Similar patterns are seen as $n$ increases.  For $n=10$, optimal performance
is obtained for almost every algorithm.  For higher dimensions, we have
 no measure of the gap to optimality.  As the dimension increases, the
 proportion of explained variation by using the same fixed cardinality
 decreases as expected.  The next subsection shows that we do not necessarily
 need to explain most of the variation in the true eigenvector in order to gain
 interpretable factors.

Results only up to a cardinality level of 250 variables are displayed because our goal is
simply to compare the different algorithms.  All algorithms, except for simple thresholding,
perform very similarly.  These figures do not sufficiently display the story, so we describe
the similar pattern that occurs.  The approximate greedy algorithm does best for smallest
cardinalities, then the expectation-maximization scheme dominates, and at some point the novel
$l_0$-constrained PCA scheme gives best performance at a higher level of explained variation.
For $n=50000$, we do not actually see this change yet because such little variation is
explained with only 250 variables.  Furthermore, it is important to notice that the thresholded solution is consistently and greatly outperformed by all other methods, suggesting that the performance results are enhanced via the conditional gradient algorithm. These experiments simply show that these algorithms offer
very similar performance, and hence we next compare them computationally.

\begin{figure}[h!] \begin{center}
  \begin{tabular} {cc}
     \psfrag{title}[t][t]{\small{$m=6,n=10$}}
     \psfrag{time}[b]{\small{Time (seconds)}}
     \psfrag{k}[t]{\small{Sparsity}}
 \includegraphics[width=0.49\textwidth]{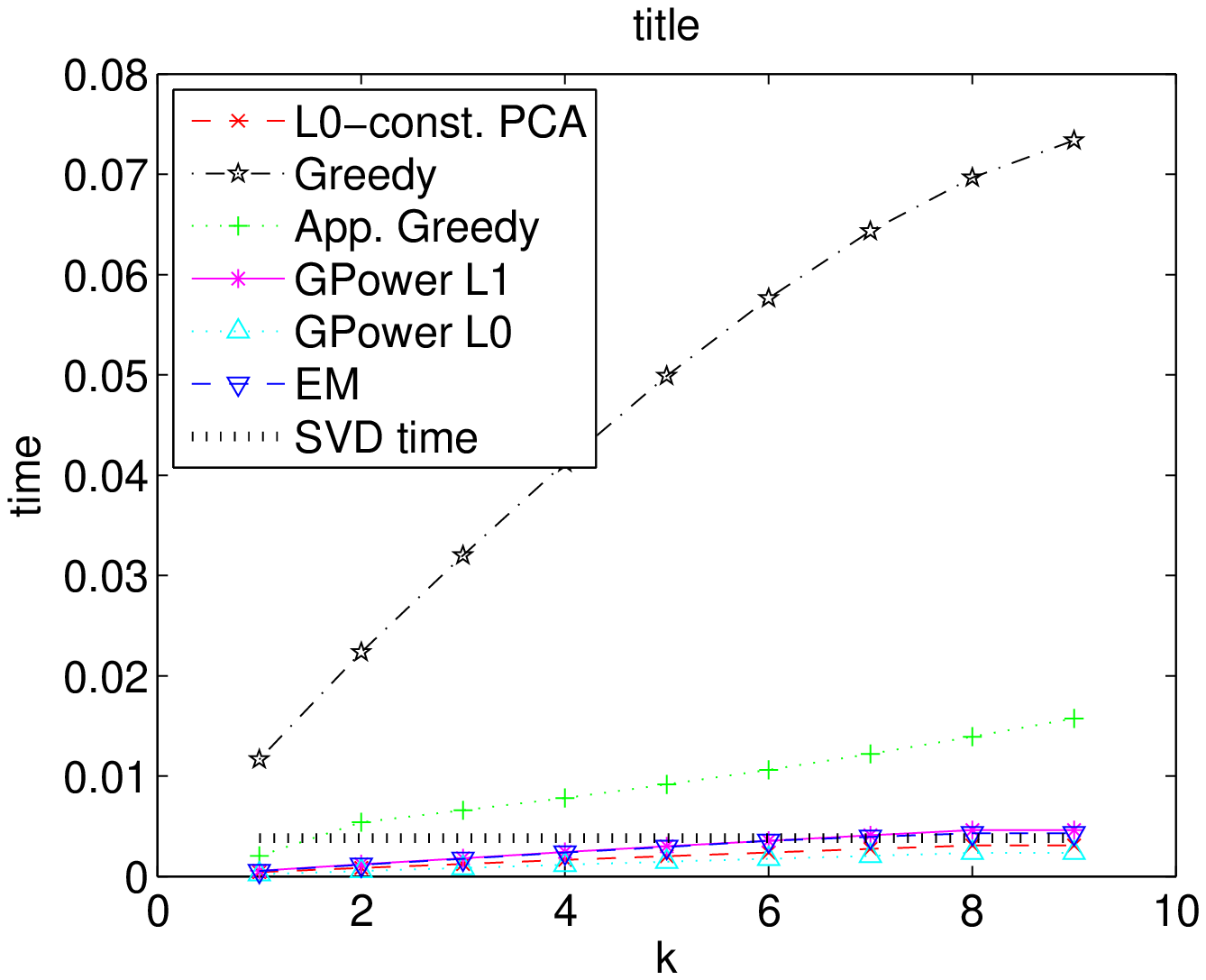}&
      \psfrag{title}[t][t]{\small{$m=150,n=5000$}}
     \psfrag{time}[b]{\small{Time (seconds)}}
     \psfrag{k}[t]{\small{Sparsity}}    \includegraphics[width=0.49\textwidth]{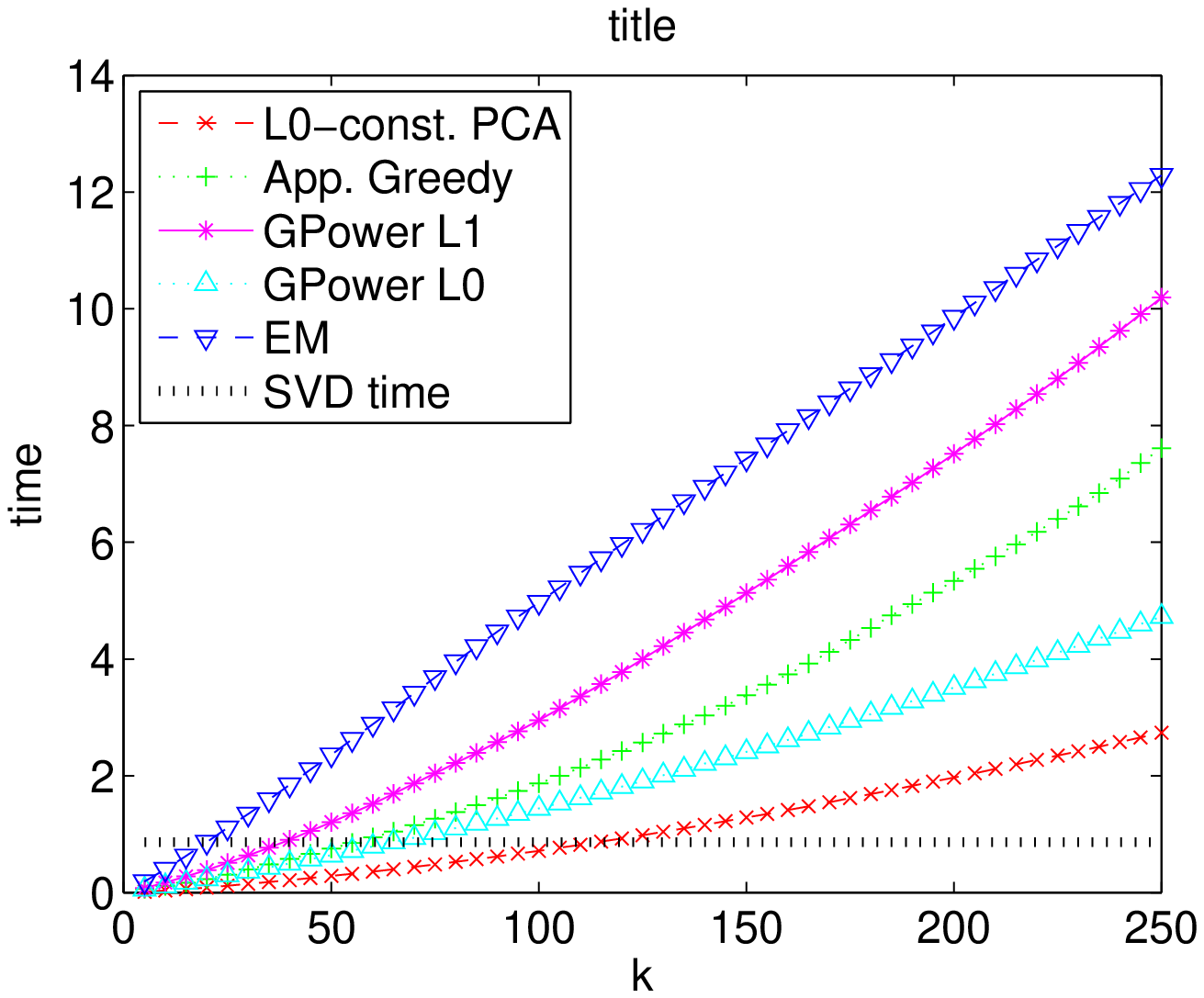}\\\\
     \psfrag{title}[t][t]{\small{$m=150,n=10000$}}
     \psfrag{time}[b]{\small{Time (seconds)}}
     \psfrag{k}[t]{\small{Sparsity}}
\includegraphics[width=0.49\textwidth]{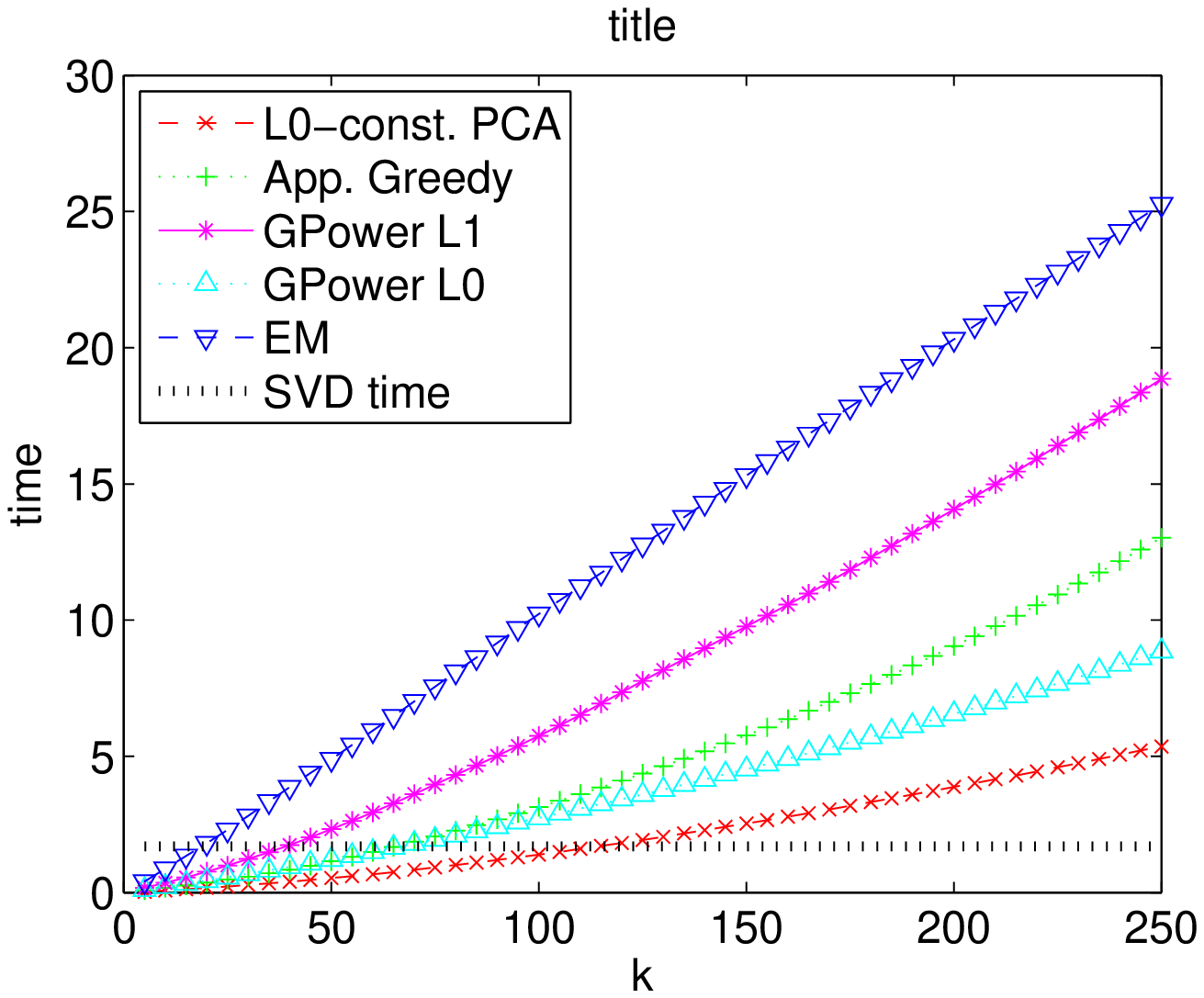}&
     \psfrag{title}[t][t]{\small{$m=150,n=50000$}}
     \psfrag{time}[b]{\small{Time (seconds)}}
     \psfrag{k}[t]{\small{Sparsity}}
\includegraphics[width=0.49\textwidth]{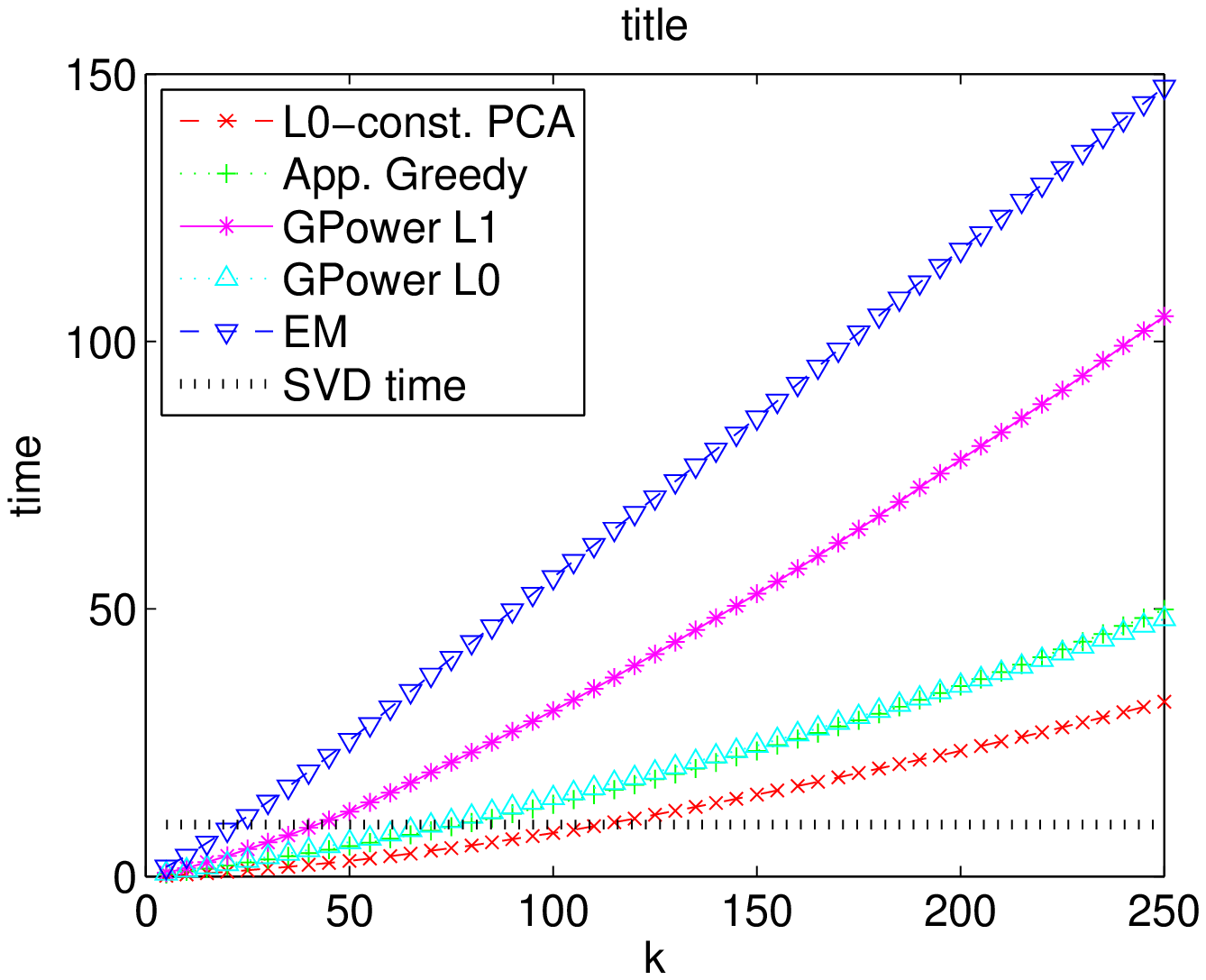}
\end{tabular}
\caption{Plots show the average cumulative time to produce the sparse eigenvectors of $F^TF$ found by several algorithms, with  $F$  an $m\times n$ matrix with $F_{ij}\sim N(0,1/m)$.  Sparsities of $2,\ldots,9$ are computed for $n=10$ and $5,10,\ldots,250$ for the remaining experiments.  Note that cumulative times are given, i.e., the time to calculate the vector with 30 nonzeros adds up the time to compute vectors with $5,10,\ldots,30$ nonzeros, in order to compare with the approximate greedy method.  The svdTime is the time required to compute the principal eigenvector of $F^TF$ which is used to compute an initial solution for $l_0$-constrained PCA.  100 simulations are used to produce all results.}
\label{fig:timeGauss}\end{center}\end{figure}

Figure \ref{fig:timeGauss} displays the computational time comparing the various algorithms for the different dimensions.   Firstly, note that for penalized PCA problems (GPowerL0 and GPowerL1), a parameter must be tuned in order to achieve the desired sparsity.  Figures here do not account for time spent tuning parameters.  Secondly, the greedy algorithms must compute the greedy solution at all sparsity levels of $1,\ldots,250$ in order to obtain the solution with 250 variables, and hence the time displayed to compute greedy solutions is the cumulative time.  Time for each algorithm is thus also taken as the cumulative time, albeit for all others it is the cumulative time to obtain a solution with sparsity levels of $5,10,\ldots,250$.  The novel $l_0$-constrained algorithm requires an initial solution which we take as thresholded solution of the true principal eigenvector.  The time to obtain that initial solution is marked as svdTime, however it need only be computed once for the entire path.

For $n=10$, the exact greedy algorithm is clearly the most expensive, requiring $n$ maximum eigenvalue computations per iteration.  For higher dimensions, the same pattern occurs.  The expectation-maximization scheme requires the most time because the scheme implicitly solves a penalized problem and thus also implicitly tunes a parameter.  GPowerL1 is surprisingly (since the tuning time is not included) next.  Despite being cheap, it requires more iterations than other methods to converge.  The approximate greedy algorithm follows, and is expected to be (relatively) computationally expensive because of the maximum eigenvalue computed at each iteration.  This is followed by GPowerL0 and finally by the cheapest scheme, the novel $l_0$-constrained PCA iteration.

We now discuss the advantages and disadvantages of the different schemes.
Clearly, if the sparsity is known (or the sparsities desired is much less than the full path),
the approximate greedy algorithm is much more computationally expensive. Comparing
the other cheap schemes, the $l_0$-constrained PCA scheme is cheapest (given the initial solution).
The disadvantage of the penalized schemes (GPowerL1 and GPowerL0) is that they must be tuned which is
 computationally very expensive (not shown).
   Warmstarting could be used, for example, by initializing for $k=10$ based on the solution
   to $k=5$ rather than from the thesholded solution.  Thus, if the desired sparsity is known,
   the $l_0$-constrained PCA scheme is clearly the algorithm to use.  If not,
   then all of the algorithms are cheap, offer similar performance,
   and can be used to derive a path of sparse solutions.

\subsection{Republicans or Democrats: What is the Difference?}\label{ss:experiments_text}
We consider here text data based on all
State of the Union addresses from 1790-2011.
Transcripts are available at \emph{http://stateoftheunion.onetwothree.net}
where other interesting analyses of this data are also done.
Here, sparse PCA is used to further analyze these historical speeches.
Questions one might ask relate to how the language in speeches has changed
from George Washington through Barak Obama or how the relevant issues divide the different presidents.
After taking the stems of words and removing commonly used \emph{stopwords}, we created a bag-of-words
data set based on all remaining words, leaving 12953 words (i.e., $n=12953$ for this example).  Our data matrix here is $B\in\reals^{m\times n}$ where $m$ is the number of speeches and $B_{ij}$ is the number of times the $j^{th}$ word occurs in the $i^{th}$ State of the Union address. We analyze two different sample sizes: using all speeches from 1790-2011 ($m=225$) and only speeches from 1982-2011 ($m=31$).  The rows of $B$ are normalized so that each speech is of the same length.  The following results are for PCA and sparse PCA performed on the covariance matrix $A=B^TB$.

\begin{figure}[h!] \begin{center}
  \begin{tabular} {cc}
     \psfrag{title}[t][t]{\small{1790-2011}}
     \psfrag{perc}[b]{\small{Proportion of explained variance}}
     \psfrag{k}[t]{\small{Sparsity}}     \includegraphics[width=0.49\textwidth]{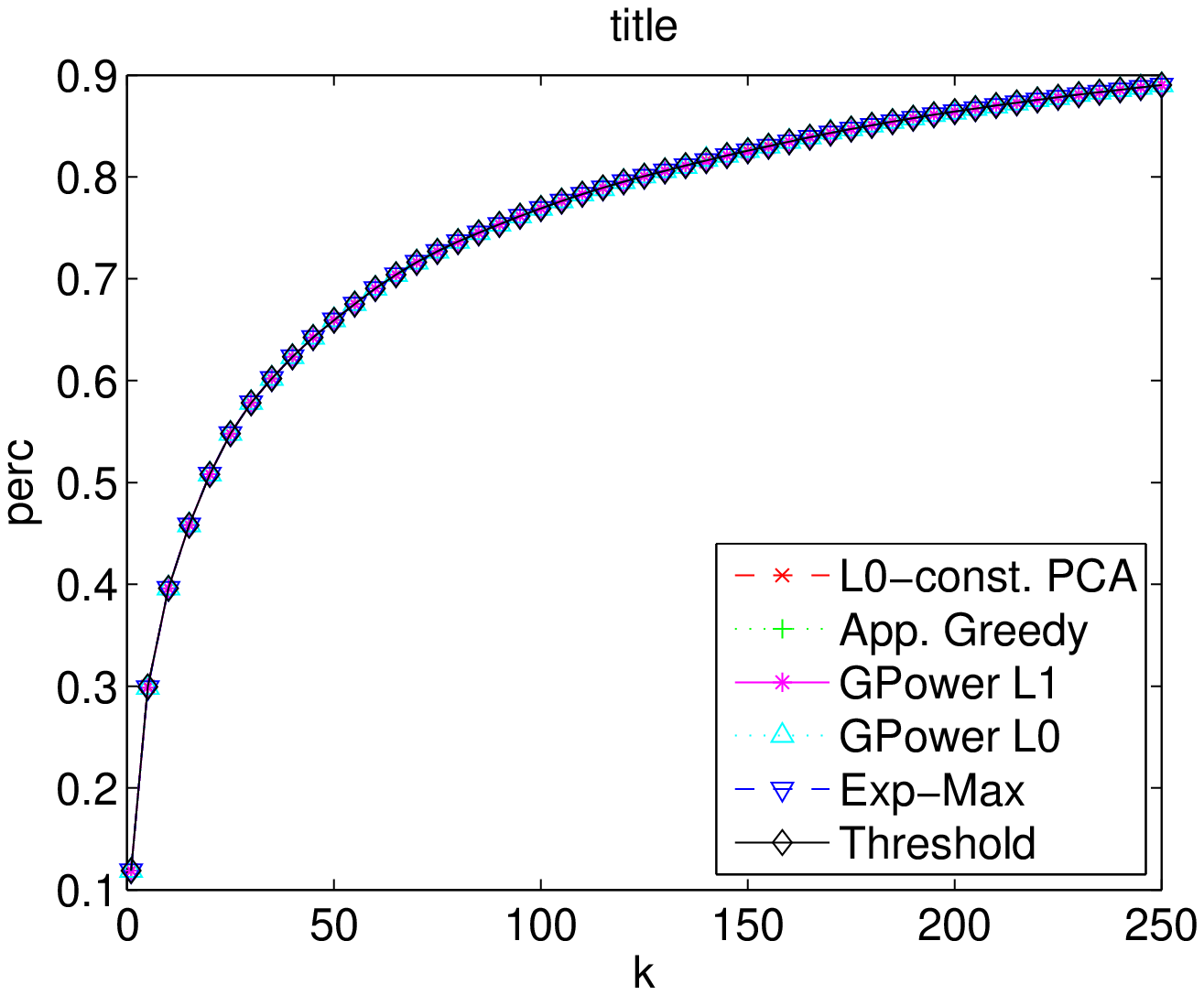}&
     \psfrag{title}[t][t]{\small{1790-2011}}
     \psfrag{time}[b]{\small{Time (seconds)}}
     \psfrag{k}[t]{\small{Sparsity}}
\includegraphics[width=0.49\textwidth]{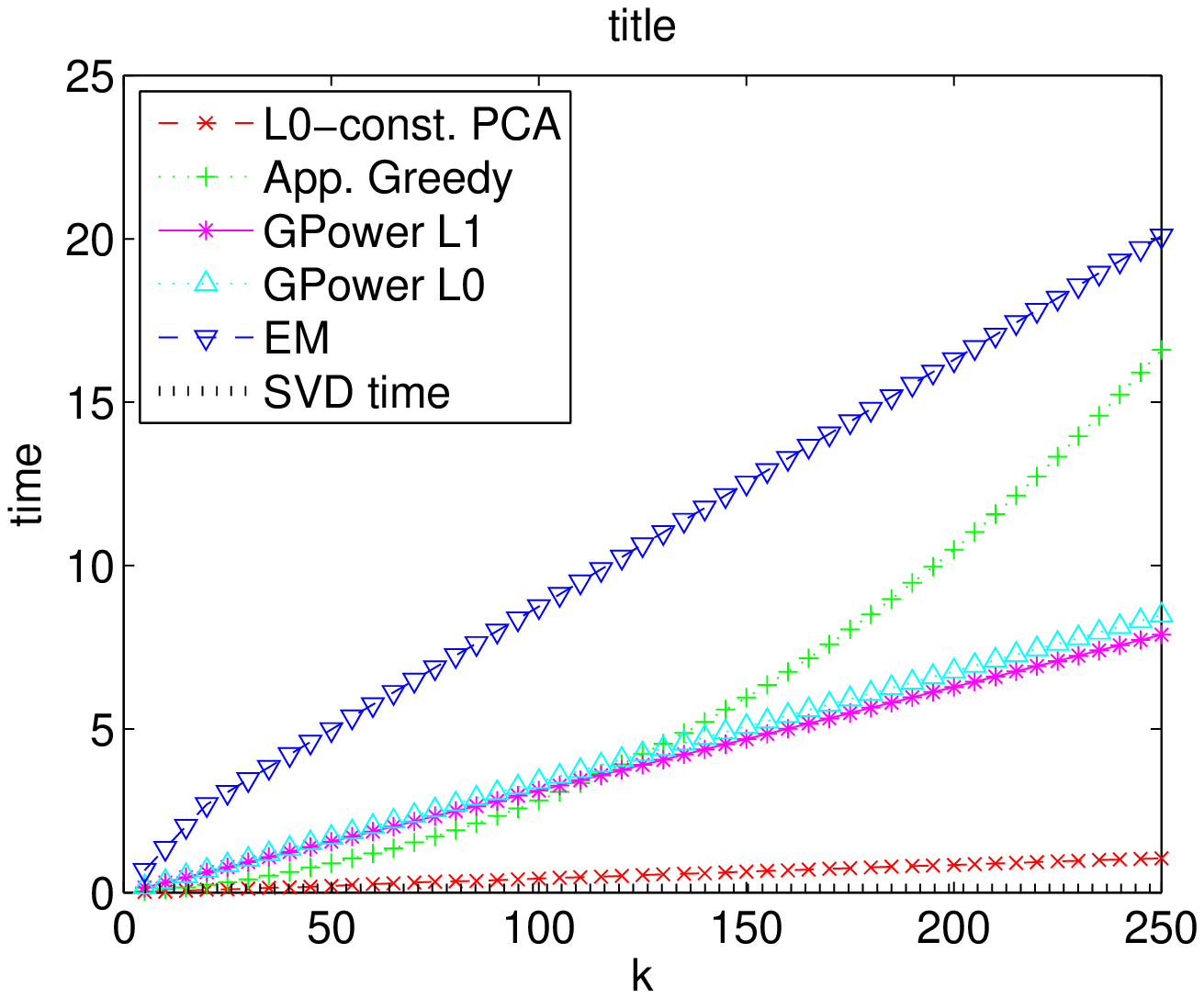}\\\\
     \psfrag{title}[t][t]{\small{1982-2011}}
     \psfrag{perc}[b]{\small{Proportion of explained variance}}
     \psfrag{k}[t]{\small{Sparsity}}     \includegraphics[width=0.49\textwidth]{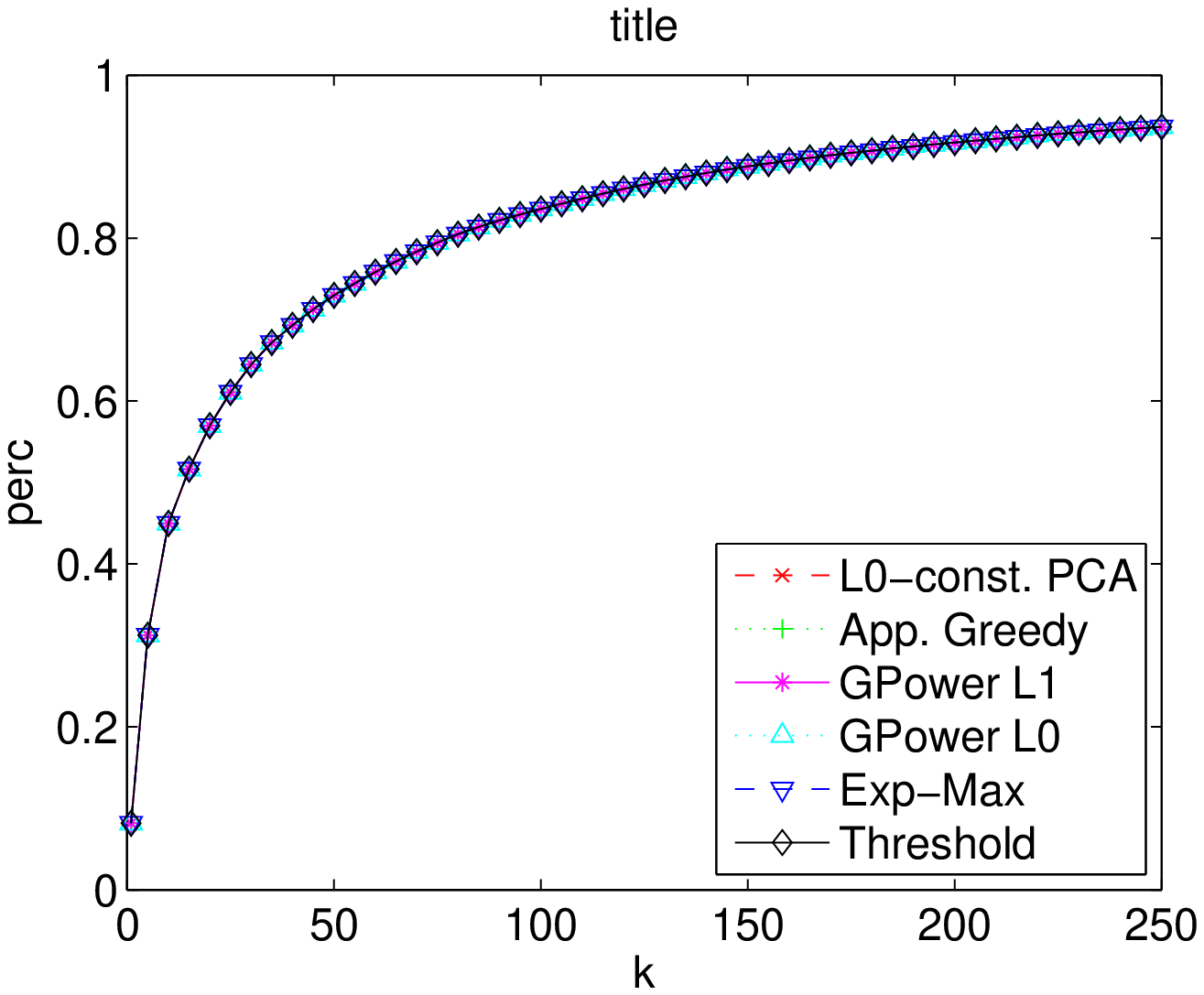}&
     \psfrag{title}[t][t]{\small{1982-2011}}
     \psfrag{time}[b]{\small{Time (seconds)}}
     \psfrag{k}[t]{\small{Sparsity}}
\includegraphics[width=0.49\textwidth]{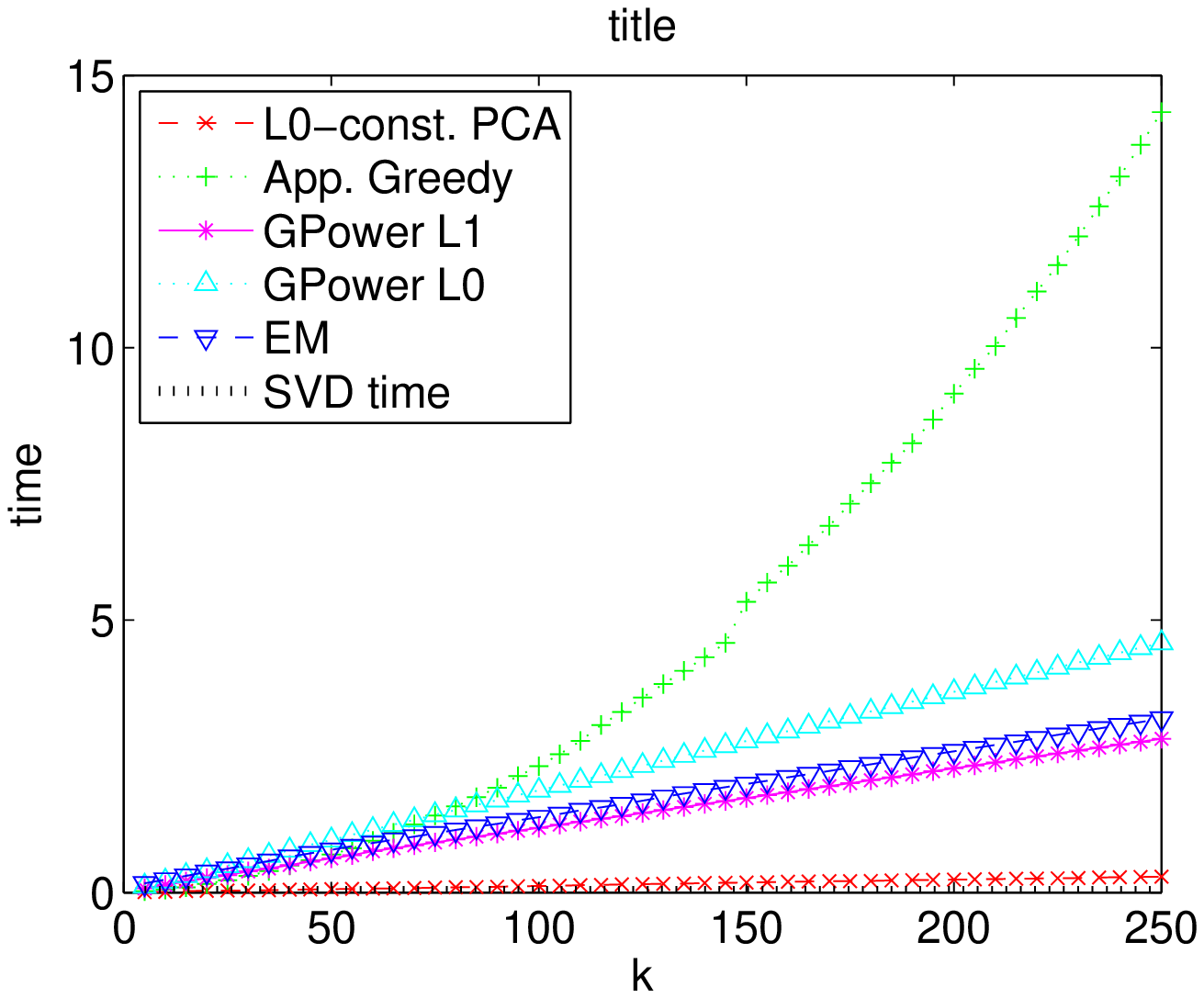}
  \end{tabular}
\caption{The left plots shows the average percent of variance explained by the sparse eigenvectors and the right plots show the computational time to run the algorithms.  Data is from the text of State of the Union addresses.  Two different numbers of samples are used: all addresses from 1790-2011 and just those from 1982-2011.}
\label{fig:stateofunion_stats}\end{center}\end{figure}

Figure \ref{fig:stateofunion_stats} displays the performance of the various algorithms on the text data set using the same measure as with random data above.  For this high-dimensional data set, much fewer variables are needed to explain more variation relative to what was observed with the random matrices.  Another major difference is that each algorithm gives exactly the same solution; even just taking the thresholded solution gives the same solution.  This leads to postulate that real data often might contain a structure that makes it rather easy to solve (still, of course, we have no results on solution quality).  Note that the $l_0$-constrained scheme seems to be the cheapest algorithm here because, starting at the thresholded solution, it required one iteration to achieve convergence!  Despite the simplicity of obtaining sparse solutions for this data, we continue to show what can be learned using this tool.  The goal is to show that sparse factors offer interpretability that cannot be learned from using all 12953 variables.

\begin{figure}[h!] \begin{center}
  \begin{tabular} {cc}
     \psfrag{pca}[t]{\small{PCA}}
     \psfrag{fac2}[b]{\small{Factor 2}}
     \psfrag{fac1}[t]{\small{Factor 1}}
     \includegraphics[width=0.49\textwidth]{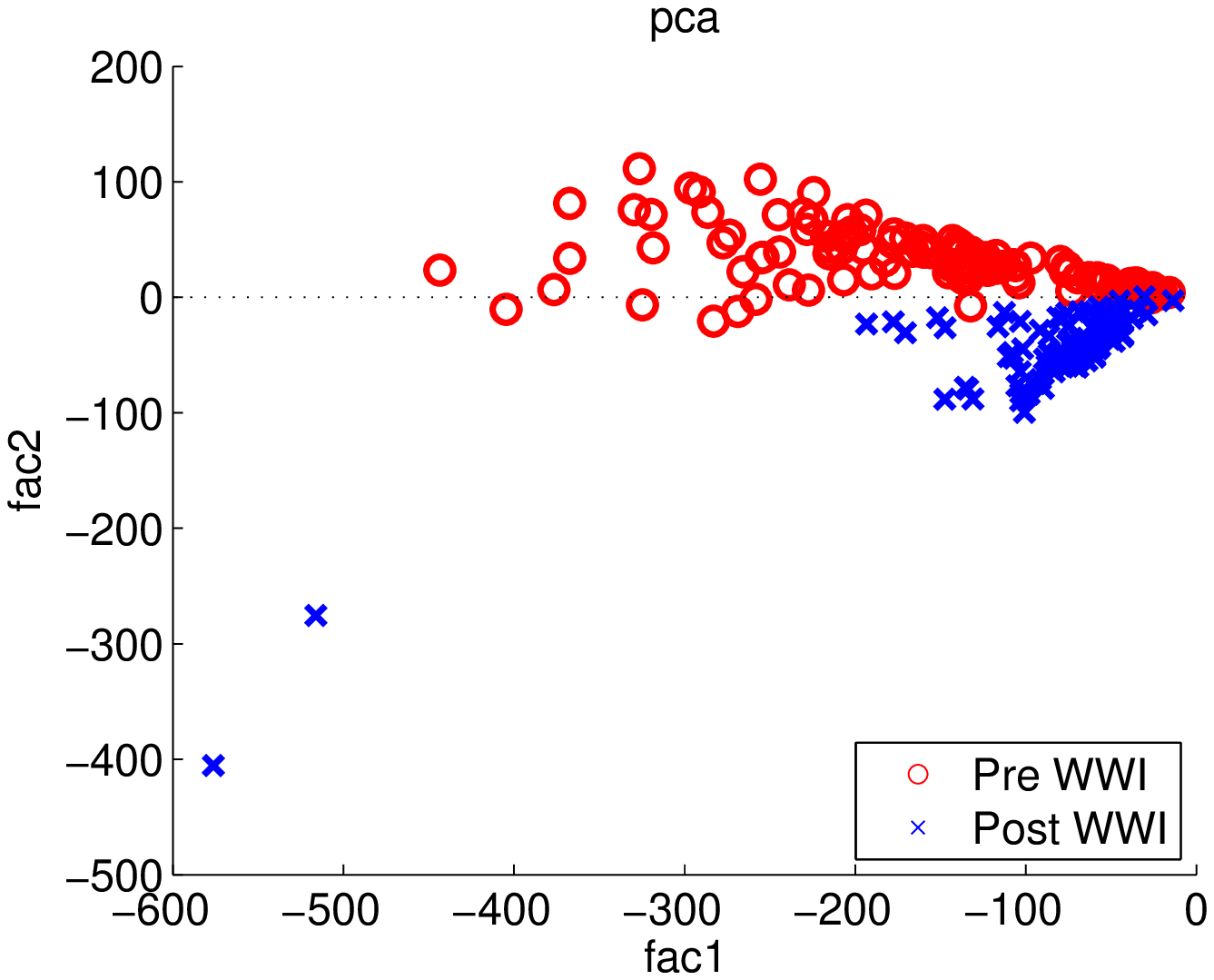}&
     \psfrag{spca}[t]{\small{Sparse PCA}}
     \psfrag{fac2}[b]{\small{Factor 2}}
     \psfrag{fac1}[t]{\small{Factor 1}}
     \includegraphics[width=0.49\textwidth]{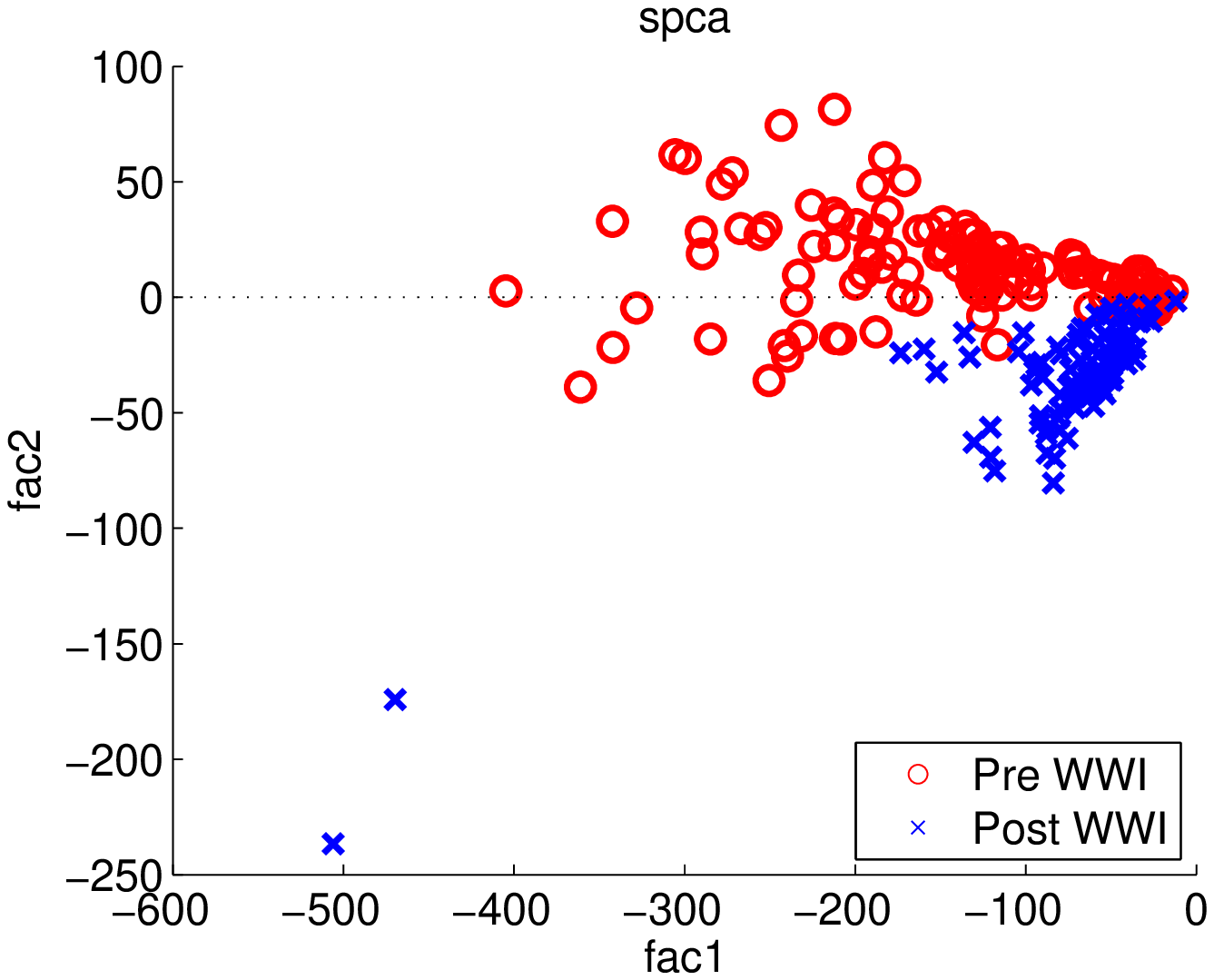}
  \end{tabular}
\caption{The left plot shows the text of all State of the Union addresses from 1790-2011 reduced from 12953 to 2 dimensions using PCA.  The right plot shows the same data reduced to 2 dimensions using $l_0$-constrained PCA.  Factor 1 is a function of 150 words, while Factor 2 is a function of only 15 words.}
\label{fig:ww1}\end{center}\end{figure}

Figure \ref{fig:ww1} (left) shows the result of projecting the data on the first two principal eigenvectors\footnote{We use projection deflation, as described in \cite{Mack2009}, to obtain multiple sparse factors.}.  The second factor clearly clusters the speeches into two groups (roughly into positive versus negative coordinates in the second factor).  Examination of these two groups shows a chronological pattern which as seen in the figure clusters those speeches that occurred before (and during) World War I with those speeches that were made after the war.  Figure \ref{fig:ww1} (right) shows the data projected on sparse factors giving very a similar illustration.  Factors 1 and 2 use 150 and 15 variables, respectively.  Table \ref{table:ww1} displays the words from the sparse second factor and their sign.  We \emph{roughly} associate the positively weighted words with speeches before the war and negatively weighted words with speeches after the war.  Then one might interpret that, before World War I, presidents spoke about the United States of America as a collection of united states, but afterwords, spoke about the country as one american nation that faced issues as a whole.

\begin{table}[h!]
\begin{center}
\begin{tabular}{|c|c|}\hline
Factor 2&Sign\\ \hline
govern&+ \\ \hline
state&+ \\ \hline
unite&+ \\ \hline
american&- \\ \hline
econom&- \\ \hline
feder&- \\ \hline
help&- \\ \hline
million&- \\ \hline
more&- \\ \hline
nation&- \\ \hline
new&- \\ \hline
program&- \\ \hline
work&- \\ \hline
world&- \\ \hline
year&- \\ \hline
\end{tabular}
\end{center}
\caption{The table shows the top 15 our of 12953 words associated with the second factor derived from $l_0$-constrained PCA on the text of all State of the Union addresses from 1790-2011.  The words for Factor 2 of thresholded PCA are almost identical.  A first factor with 150 nonzero entries was deflated from the data.  The signs of the words are given to show what drives the difference between the clusters in Figure \ref{fig:ww1}.  Note that words here are \emph{word stems}, i.e. the words \emph{program} and its plural \emph{programs} are both represented by \emph{program}.} \label{table:ww1}
\end{table}

\begin{figure}[h!] \begin{center}
  \begin{tabular} {cc}
     \psfrag{pca}[t]{\small{PCA}}
     \psfrag{fac2}[b]{\small{Factor 2}}
     \psfrag{fac1}[t]{\small{Factor 1}}     \includegraphics[width=0.49\textwidth]{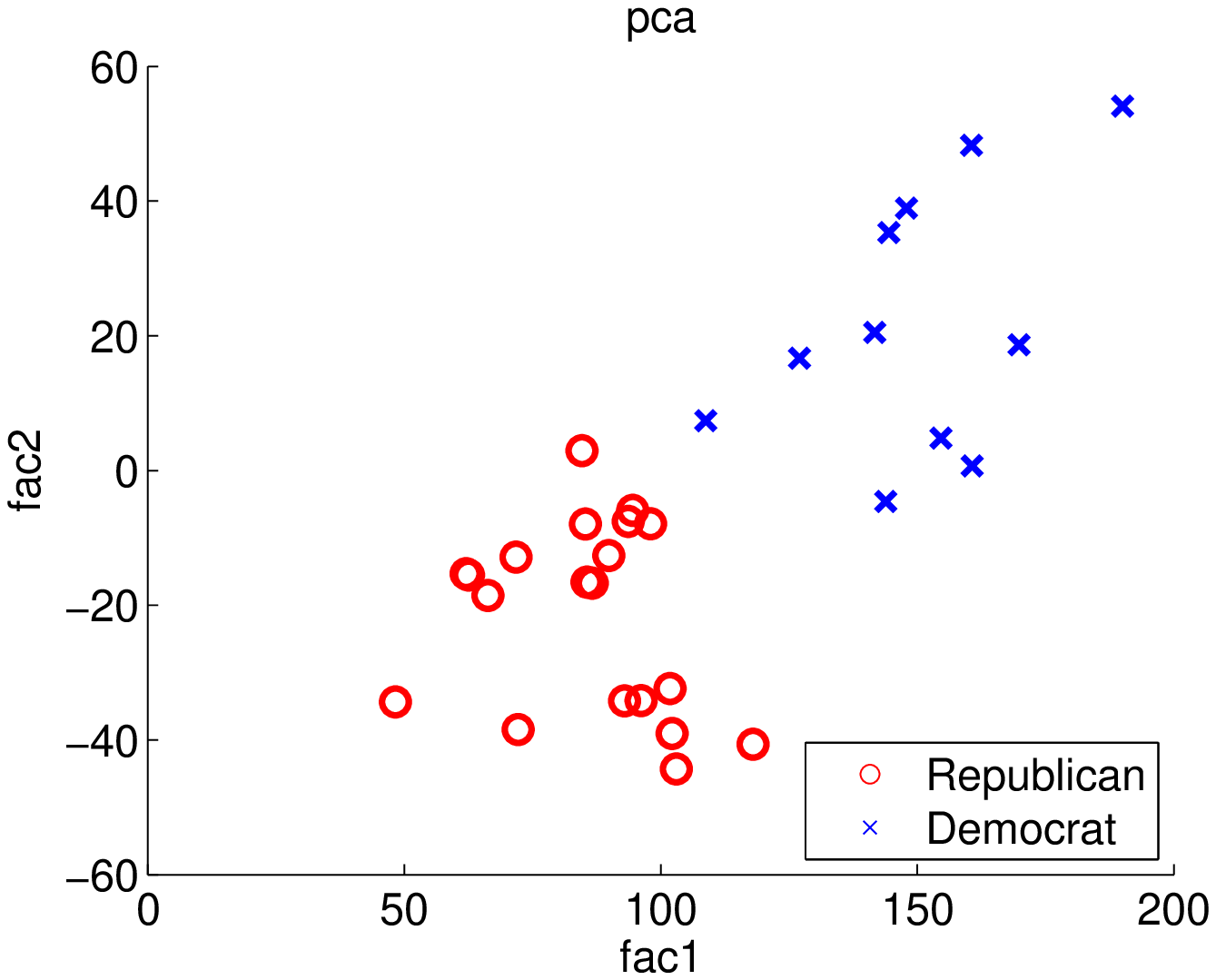}&
     \psfrag{spca}[t]{\small{Sparse PCA}}
     \psfrag{fac2}[b]{\small{Factor 2}}
     \psfrag{fac1}[t]{\small{Factor 1}}
\includegraphics[width=0.49\textwidth]{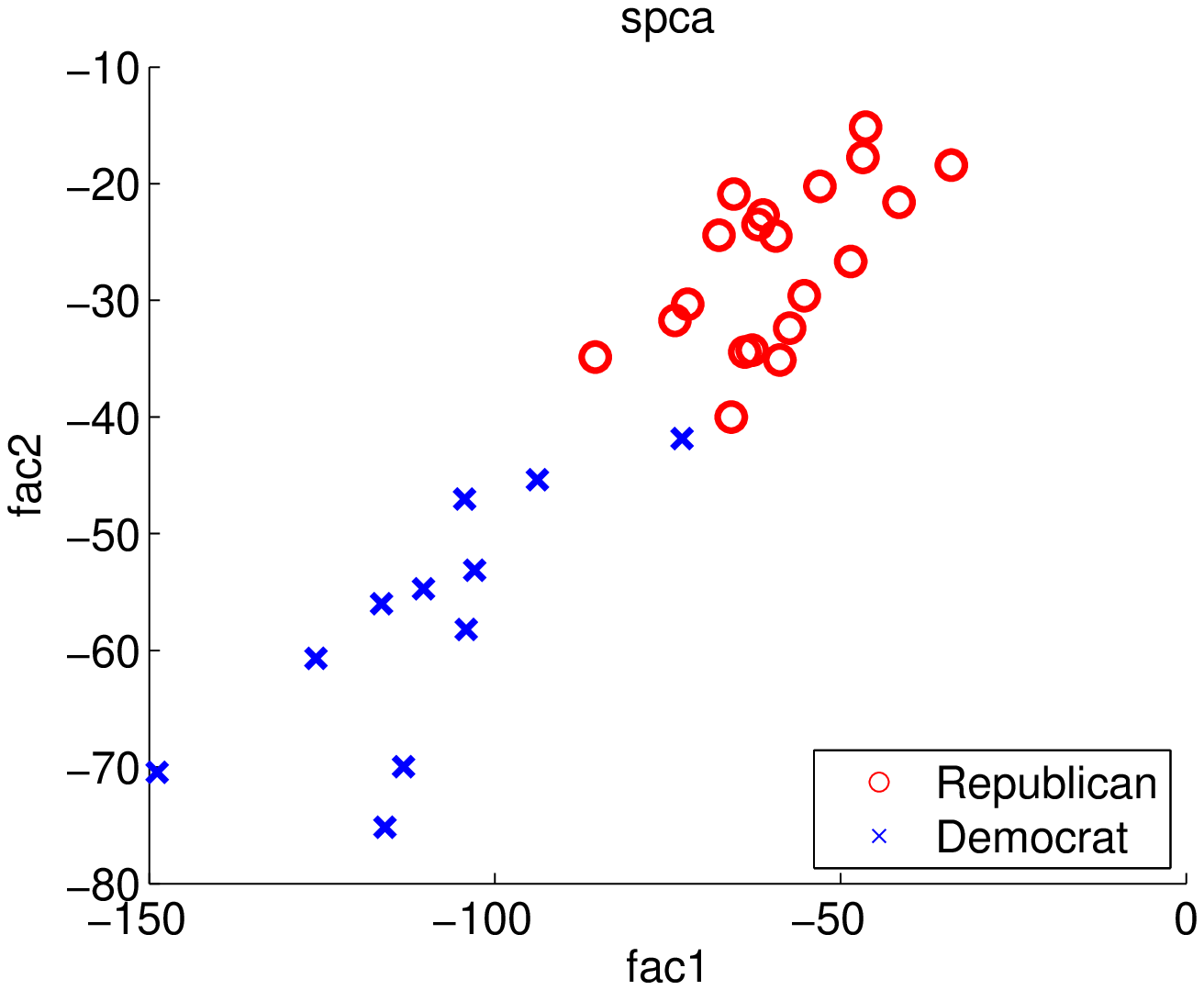}
  \end{tabular}
\caption{The left plot shows the text of 31 State of the Union addresses from 1982-2011 reduced from 12953 to 2 dimensions using PCA.  The right plot shows the same data reduced to 2 dimensions using $l_0$-constrained PCA where both factors are functions of exactly 15 words.}
\label{fig:parties}\end{center}\end{figure}

Figure \ref{fig:parties} next shows a similar analysis using only speeches from 1982-2011.  A clear distinction between republicans and democrats was discovered.  Again, sparse PCA is used to interpret the factors with 15 variables each.  Table \ref{table:parties} shows the most important words discovered by PCA (using thresholding) and sparse PCA for the top 3 factors.  The first factor gives the same words for both analyses and no clear interpretation is seen.  The second factors are more interesting.  The thresholded PCA factor clearly relates to international security issues.  The sparse PCA factor, however, clearly focuses on domestic issues, e.g., health-care and education.  Differences occur because PCA deflates with the true eigenvectors while sparse PCA deflates with the sparse factors.  In any case, these are clearly issues that divide republicans and democrats.  The third thresholded PCA factor seems to be related to educational reforms and funding them.  The third sparse PCA factor is clearly related to fiscal policies and the economy.

\begin{table}[h!]
\begin{center}
\begin{tabular}{|c|c|c|c|c|c|c|}
\cline{1-3}\cline{5-7}
\multicolumn{3}{|c|}{PCA Factors 1-3}&&\multicolumn{3}{|c|}{Sparse PCA Factors 1-3} \\ \cline{1-3}\cline{5-7}
  1 & 2 & 3 & & 1 & 2 & 3 \\ \cline{1-3}\cline{5-7}
year&america&govern&&year&job&budget\\ \cline{1-3}\cline{5-7}
american&work&peopl&&american&countri&economi\\ \cline{1-3}\cline{5-7}
more&terrorist&program&&more&children&program\\ \cline{1-3}\cline{5-7}
peopl&world&school&&peopl&secur&live\\ \cline{1-3}\cline{5-7}
work&freedom&centuri&&work&famili&million\\ \cline{1-3}\cline{5-7}
america&nation&commun&&new&tonight&over\\ \cline{1-3}\cline{5-7}
new&peopl&spend&&america&care&busi\\ \cline{1-3}\cline{5-7}
nation&more&children&&nation&health&futur\\ \cline{1-3}\cline{5-7}
make&cut&america&&make&ask&plan\\ \cline{1-3}\cline{5-7}
help&iraq&new&&help&last&reform\\ \cline{1-3}\cline{5-7}
govern&year&tax&&govern&school&most\\ \cline{1-3}\cline{5-7}
world&terror&countri&&world&state&mani\\ \cline{1-3}\cline{5-7}
time&care&deficit&&time&support&respons\\ \cline{1-3}\cline{5-7}
tax&job&challeng&&congress&commun&good\\ \cline{1-3}\cline{5-7}
congress&secur&support&&tax&cut&invest\\ \cline{1-3}\cline{5-7} \end{tabular}
\end{center}
\caption{The left side of the table shows the top 15 words associated with the first 3 factors derived from PCA on the text of 31 State of the Union addresses from 1982-2011.  The right side of the table shows the top 15 words when the 3 factors are derived using $l_0$-constrained PCA. The original number of dimensions is 12953. Note that words here are \emph{word stems}, i.e. the words \emph{program} and its plural \emph{programs} are both represented by \emph{program}.} \label{table:parties}
\end{table}

\section{Concluding Remarks} \label{s:extensions}
Sparse PCA admits a simple formulation which maximizes a (usually convex) quadratic objective
subject to \emph{seemingly} simple constraints. Nonetheless, it is a difficult nonconvex
optimization problem, and a large literature has focused on various modifications of the
desired problem in order to derive simple algorithms that hopefully produce \emph{good}
approximate solutions. No gaps to the solution of the original problem are given by any of the
currently known schemes.

In this paper, we have shown that the conditional gradient algorithm ConGradU:
\begin{itemize}
\item  can be directly applied to the
original $l_0$-constrained PCA problem (\ref{eq:sparse_pca}) without any modifications to produce a very
simple scheme with low computational complexity.
\item can be successfully applied to maximizing a convex function over an arbitrary compact set and was proven to exhibit global convergence to stationary points.  Efficiency of this scheme builds on the result that, while maximizing a quadratic function over the $l_2$ unit ball with an $l_0$ constraint is a difficult problem, maximizing a linear function over the same nonconvex set is simple.
\item provides a unifying framework to derive and analyze all new and old schemes discussed
in the paper which, as we have seen, were derived from  disparate approaches in the cited
literature. As shown, all these algorithms which have been used in various applications are
special cases of ConGradU.
\end{itemize}

The overall message  is that, for some difficult problems, we can achieve the same (limited)
theoretical guarantees and practical performance using the same algorithm on modified
(seemingly easier) problems as we can on the original difficult problem.  Furthermore, all of these algorithms emerging from ConGradU give similar performance in practice and
with similar complexities.

We conclude by showing that the same tools we have used for applying the ConGradU algorithm to sparse PCA can readily be applied to other sparsity-constrained statistical tools, e.g., sparse Singular Value Decompositions, sparse Canonical Correlation Analysis, and sparse nonnegative PCA.
Note that our comments below about $l_0$-constrained problems can also be extended to the corresponding $l_1$-constrained and $l_0/l_1$-penalized problems.\\

\noindent{\bf Sparse Singular Value Decomposition (SVD)}\\
Sparse SVD solves the problem
\begin{equation} \label{eq:sparse_svd}
\max{\{x^TBy : \|x\|_2=1, \|y\|_2=1, \|x\|_0\leq k_1, \|y\|_0\leq k_2,x\in\reals^m,y\in\reals^n\}}
\end{equation}
where $B\in\reals^{m\times n}$ is the data matrix and $k_1, k_2$ are positive integers.  Note that this is equivalent to $l_0$-constrained PCA when $x=y$ and $B\in\symm^n$.  \cite{Witt2009} considered a relaxation to this problem as described in Section \ref{ss:l1constrained_pca}.  They relaxed the $l_0$ constraints on $x$ and $y$ with $l_1$ constraints (and relaxed equalities to inequalities) and applied an alternating maximization scheme, where each optimization problem is easily solved via Proposition \ref{prop:lin_over_l2ball_l1ball}.  Following exactly the approach we suggested for $l_0$-constrained PCA, there is no need to relax the $l_0$ constraints.  Proposition \ref{prop:lin_over_l2ball_card} can be used to solve directly the alternating optimization problems, giving rise to a simple algorithm for the $l_0$-constrained SVD problem (\ref{eq:sparse_svd}).\\

\noindent{\bf Sparse Canonical Correlation Analysis (CCA)}\\
Sparse CCA solves the problem
\begin{equation} \label{eq:sparse_cca}
\max{\{x^TB^TCy : x^TB^TBx=1, y^TC^TCy=1, \|x\|_0\leq k_1, \|y\|_0\leq k_2,x\in\reals^p,y\in\reals^q\}}
\end{equation}
where $B\in\reals^{m\times p}$ and $C\in\reals^{m\times q}$ are data matrices and $k_1, k_2$ are positive integers. \cite{Witt2009} suggest that useful (i.e., interpretable) results can still be obtained by substituting the identity matrix for $B^TB$ and $C^TC$ in the constraints, resulting in the sparse SVD problem above. Rather, we propose substituting the diagonals of $B^TB$ and $C^TC$ as proxies.  Propositions \ref{prop:lin_over_l2ball_card} and \ref{prop:lin_over_l2ball_l1ball} (among others in Section \ref{s:props}) can both easily be extended to optimizing over the constraints $x^TDx=1$ and $x^TDx\leq1$, where $D$ is diagonal. A simple alternating maximization algorithm then follows for the resulting $l_0$-constrained \emph{approximate} CCA problem.\\

\noindent{\bf Sparse Nonnegative Principal Component Analysis (PCA)}\\
Sparse nonnegative PCA solves the problem
\begin{equation} \label{eq:sparse_nn_pca}
\max{\{x^TAx : \|x\|_2=1, \|x\|_0\leq k, x\ge0, x\in\reals^n\}}
\end{equation}
where $A\in\reals^{n\times n}$ is a covariance matrix and $k$ is a positive integer.  This is exactly the $l_0$-constrained PCA problem (\ref{eq:sparse_pca}) with additional nonnegativity constraints.  Simple extensions of Propositions \ref{prop:lin_over_l2ball_card} and \ref{prop:lin_over_l2ball_l1ball} lead to a simple scheme for $l_0$-constrained nonnegative PCA based on the ConGradU algorithm.

\section{Acknowledgements}
The authors thank Noureddine El Karoui for useful discussions regarding random matrices.  This research was partially supported by the United States-Israel Science Foundation under BSF Grant \#2008-100.

\small
\bibliographystyle{plain}
\bibliography{sparsePCA}
\end{document}